\documentclass[a4paper,1p]{elsarticle}
\usepackage{amsmath,amssymb,anysize,float} 
\usepackage[normalem]{ulem}
\numberwithin{equation}{section}

\newtheorem{theorem}[equation]{Theorem}
\newtheorem{lemma}[equation]{Lemma}
\newtheorem{cor}[equation]{Corollary}

\newtheorem{hyp}[equation]{Hypothesis}
\newtheorem{notn}[equation]{Notation}
\newproof{proof}{Proof}

\newcommand\PG{\mathsf{PG}}
\newcommand\Q{{\mathcal Q}}
\newcommand\soc{\mathsf{soc}}

\DeclareMathOperator{\Wr}{wr}
\DeclareMathOperator{\Sym}{Sym}

\renewcommand\le{\leqslant}
\renewcommand\ge{\geqslant}

\begin{document}

\begin{frontmatter}

\title{Generalised quadrangles with a group of automorphisms acting primitively on points and lines}
 
\author[uwa]{John Bamberg}
\ead{John.Bamberg@uwa.edu.au}

\author[uwa]{Michael Giudici}
\ead{Michael.Giudici@uwa.edu.au}

\author[lethbridge]{Joy Morris}
\ead{joy.morris@uleth.ca}

\author[uwa]{Gordon F. Royle}
\ead{Gordon.Royle@uwa.edu.au}

\author[uwa]{Pablo Spiga}
\ead{Pablo.Spiga@uwa.edu.au}

\address[uwa]{
   The Centre for the Mathematics of Symmetry and Computation,\\
   School of Mathematics and Statistics,\\
   The University of Western Australia,\\
   35 Stirling Highway, Crawley, W. A. 6009,  Australia}

\address[lethbridge]{
Department of Mathematics and Computer Science,\\ University of Lethbridge,\\ Lethbridge, AB. T1K 3M4. Canada}


% 2008 MSC numbers
% 51E12 generalised quadrangles, generalised polygons
% 05B25 Finite geometries
% 20B15 Primitive groups

\begin{abstract}
We show that if $G$ is a group of automorphisms of a thick finite generalised quadrangle
$\mathcal{Q}$ acting primitively on both the points and lines of $\mathcal{Q}$, then $G$ is almost simple. 
Moreover, if $G$ is also flag-transitive then $G$ is of Lie type.
\end{abstract}

\end{frontmatter}

\section{Introduction}\label{section:introduction}

A (finite) \emph{generalised $d$-gon} is a finite point-line geometry whose bipartite incidence graph has diameter $d$ and girth $2d$. Despite its simplicity, this definition includes some of the most fundamental objects studied in finite geometry, including projective planes ($d=3$) and generalised quadrangles ($d=4$).
Generalised polygons were introduced by Tits \cite{Tits:1959fk} in an attempt to find geometric models for simple groups of Lie type. In particular, the group $\mathsf{PSL}(3,q)$ is admitted by the Desarguesian projective plane $\PG(2,q)$, the groups $\mathsf{PSp}(4,q)$, $\mathsf{PSU}(4,q)$, $\mathsf{PSU}(5,q)$ are admitted by certain generalised quadrangles and $\mathsf{G}_2(q)$, $\,^3\mathsf{D}_4(q)$, and $\,^2\mathsf{F}_4(q)$ arise as automorphism groups of two generalised hexagons ($d=6$) and a generalised octagon ($d=8$) respectively. We call these generalised polygons the \emph{classical} generalised polygons, and observe that in all these cases, the groups act primitively on both points and lines, and transitively on the \emph{flags} (i.e., the incident point-line pairs) of the generalised polygon. The automorphism groups of the classical generalised polygons also act \emph{distance-transitively} on the points of these geometries and Buekenhout \& Van Maldeghem \cite{Buekenhout:1994vn} showed that this property \emph{almost} characterises the classical generalised polygons. In particular, they showed that a generalised quadrangle constructed from a hyperoval in $\PG(2,4)$ (Payne \cite{Payne:1990fk}) is the only non-classical generalised polygon with a group acting distance-transitively on points.

Distance-transitivity is a very strong symmetry condition and various authors have considered the extent to which weaker (or just different) symmetry conditions characterise the classical generalised polygons. To exclude trivial cases, we require that the geometry is \emph{thick} (i.e. each line contains at least three points and each point lies on at least three lines) in which case there are constants $s$ and $t$ such that each line contains $s+1$ points, each point lies on $t+1$ lines, and the generalised polygon is said to have \emph{order} $(s,t)$. The celebrated theorem of Feit and Higman \cite{Feit:1964rt} shows that a thick generalised $d$-gon can only exist when $d \in \{2,3,4,6,8\}$, and as generalised $2$-gons are simply geometries whose incidence graph is complete bipartite, they can also be regarded as trivial.  This leaves four distinct types of generalised polygon that can be considered separately if necessary. For projective planes, it has long been conjectured (see Dembowski \cite{Dembowski:1997fk}) that even the mild condition of transitivity on points characterises the classical projective plane $\PG(2,q)$. However it is not even known whether the much stronger condition of flag-transitivity is sufficient. Higman and McLaughlin \cite{Higman:1961uq} showed that a group acting flag-transitively on a projective plane acts \emph{primitively} on both the points and lines of the plane, and Kantor \cite{Kantor:1987fr} showed that a group acting primitively on a non-classical projective plane contains a cyclic normal subgroup of prime order acting regularly on the points. Ultimately this leads to severe numerical restrictions on the possible size of the projective plane, and Thas \& Zagier \cite{Thas:2008fk} have shown that these restrictions are not satisfied by any non-classical projective plane with fewer than $4 \times 10^{22}$ points. Recently, Gill \cite{Gill:2007uq} proved that all minimal normal subgroups of a group acting transitively on a non-classical projective plane are elementary abelian.

Both Kantor's results on projective planes and Buekenhout and Van Maldeghem's characterisation of the generalised polygons with a group acting distance-transitively on points rely heavily on fundamental results regarding the structure of primitive permutation groups. While Buekenhout and Van Maldeghem show that primitivity is a necessary consequence of distance-transitivity, it is not necessarily the case that a flag-transitive group of automorphisms of a generalised polygon with $d \ge 4$ is primitive on either points or lines. However, by adding primitivity to the symmetry assumptions, Schneider and Van Maldeghem \cite{Schneider:2008zr} show that if a group acts flag-transitively, point-primitively, and line-primitively on a generalised hexagon ($d=6$) or generalised octagon ($d=8$), then it is an almost simple group of Lie type.

In this paper, we focus on the remaining case of \emph{generalised quadrangles}, first with only the primitivity assumptions, and prove the following main result:

\begin{theorem}\label{main1}
A group of automorphisms acting primitively on the points and lines of a finite thick generalised quadrangle is almost simple.
\end{theorem}

The proof of this result is based on the O'Nan-Scott theorem which classifies primitive groups into various types; there are a number of different ways to do this, but we use the classification into 8 types given by Praeger \cite{praegerbcc} which is summarised in Table~\ref{tab:primgroups}.  As the hypothesised automorphism group of the generalised quadrangle acts primitively on both points and lines, we consider the combinations of pairs of the 8 types according to the action of the group on points and lines respectively and, using a variety of techniques, eliminate all combinations except when the group is almost simple.

To obtain the analogue of Schneider and Van Maldeghem's result for generalised hexagons and generalised octagons, we must eliminate
the almost simple groups whose socle is a sporadic simple group or an alternating group. The sporadic case can be ruled out by numerical
considerations based on the degrees of the possible primitive actions, but the latter case causes more problems.   The classical generalised quadrangle of order $(2,2)$ has the symmetric group $S_6$ acting on it and, although $S_6$ is isomorphic to $\mathsf{PSp}(4,2)$ and
can therefore be viewed as a group of Lie type, the proof must admit this as a special case.  More importantly, we could not exclude
the remaining alternating groups without the additional symmetry hypothesis that the group acts \emph{flag-transitively}, although we suspect that this is needed only for the proof and not for the result itself. Using this we obtain the following result as a direct consequence of Corollary~\ref{final-cor-sporadics}, Theorem~\ref{A_n-intrans-done}, Corollary~\ref{cor:imprim} and Theorem~\ref{A_n-smalldegree-done}:

\begin{theorem}\label{theorem:alternatinggroups}\label{theorem:sporadics}
Let $G$ be a group of automorphisms of a finite thick generalised quadrangle $\mathcal{Q}$.  
\begin{enumerate}[(a)] 
\item If $G$ acts primitively on the points and lines of $\mathcal{Q}$, then
the socle of $G$ is not a sporadic simple group. 
\item  If $G$ acts flag-transitively and point-primitively on $\mathcal{Q}$ and the socle of $G$ is an alternating
group $A_n$ with $n\ge 5$, then $G \le S_6$ and $\mathcal{Q}$ is the unique generalised quadrangle of order $(2,2)$.
\end{enumerate}
\end{theorem}

%\begin{theorem}\label{theorem:alternatinggroups}\label{theorem:sporadics}
%Let $G$ be a group of automorphisms of a finite thick generalised quadrangle, and suppose that $G$
%acts primitively on its points and lines.  Then the socle of $G$ is not a sporadic simple group.  If
%the socle of $G$ is an alternating group, then one of the following occurs:
%\begin{enumerate}[(i)]
%\item $n=6$, $G\le S_6$ and $\mathcal{Q}$ is the unique generalised quadrangle of order $(2,2)$; or
%\item the stabiliser of a point is a primitive subgroup of $G$, and $G$ is intransitive on
%  flags of $\mathcal{Q}$.
%\end{enumerate}
%\end{theorem}

%Recall that the classical generalised quadrangles, $\mathsf{W}(3,q)$, $\mathsf{H}(3,q^2)$, $\mathsf{H}(4,q^2)$ and 
%their duals all admit an almost simple group of Lie type acting primitively on the points and lines and transitively on the flags is flag-transitive. 
% and 

Combining our results for generalised quadrangles with the previously mentioned
results for other generalised polygons yields the following unified result.  

\begin{cor}
A group of automorphisms acting primitively on the points and lines, and transitively on the flags of a finite thick generalised polygon is almost simple of Lie type or, possibly, acting on a non-classical projective plane with at least $4 \times 10^{22}$ points. 
\end{cor}

There are two further known generalised quadrangles with a group acting transitively on flags, namely the 
generalised quadrangles of order $(3,5)$ and $(15,17)$ arising from transitive
hyperovals in $\PG(2,4)$ and $\PG(2,16)$ respectively. The corresponding groups act primitively on points, but imprimitively on lines.

Our proof of these results proceeds by considering all possible pairs of types of primitive action on the points and lines respectively. The types and which pairs of types can occur together are described in Section~\ref{section:background}, along with a number of other background results on generalised quadrangles. In Section~\ref{section:asorpa} we eliminate almost all possible pairs of types by using knowledge of the primitive actions, in particular semiregular elements in these actions. The results of this section lead to the conclusion that either the action is almost simple on both points and lines (i.e. the case that cannot be eliminated as it covers the classical examples) or is of type PA (product action) on either points or lines.  Section~\ref{section:productaction} is devoted to investigating the PA type in detail where, by duality, the points can be identified as elements of a Cartesian product and the group as a subgroup of a wreath product. Arguments relating collinearity to the product structure of the point-set and their numerical consequences are used to rule out all the possible types of primitive action on lines that can be paired with the action on points. By the end of this section, the only remaining possibility is that the group is almost simple, thus proving Theorem~\ref{main1}. 
The final two sections of the paper reduce the possibilities to almost simple groups of Lie type, first eliminating almost simple groups whose socle is an alternating group in Section~\ref{section:alternating} and then those whose socle is sporadic in Section~\ref{section:sporadics}. The former requires a rather long and intricate argument where a critical lemma (Lemma~\ref{lemma:power6}) bounding the size of the group needs the additional hypothesis of flag-transitivity. By contrast, Section~\ref{section:sporadics} dealing with the sporadic groups is a straightforward, though tedious, analysis based on comparing the known degrees of the primitive actions of almost simple groups whose socle is sporadic with the possible numbers of points of a generalised quadrangle. The only possibility that survives this test is that the Rudvalis group might act primitively on the points and lines of a generalised quadrangle of order $(57,57)$ but this is quickly ruled out by consideration of the subdegrees of this action.

%%%%%%%%%%%%%%%%%%%%%%%%%%
%
% Background and definitions
%
%%%%%%%%%%%%%%%%%%%%%%%%%%

\section{Background and definitions}\label{section:background}

\subsection{Generalised quadrangles}

The requirement that the incidence graph of a generalised quadrangle have diameter $4$ and girth $8$
is equivalent to the following definition: (i) every two points lie on at most one line, and (ii)
given a point $p$ and a line $L$, there is a unique point $q$ on $L$ that is collinear with $p$. The
second condition is equivalent to there being no triangles in the geometry.  If there is a line
containing at least three points or if there is a point on at least three lines, then there are
constants $s$ and $t$ such that each line is incident with $s+1$ points, and each point is incident
with $t+1$ lines. Such a generalised quadrangle is said to have \textit{order} $(s,t)$, and hence
its point-line dual is a generalised quadrangle of order $(t,s)$.  The number of points of a generalised quadrangle of order $(s,t)$ is $(s+1)(st+1)$. The \textit{collinearity graph} of
a generalised quadrangle has the points of the generalised quadrangle as its vertices with two vertices being adjacent if
and only if they are collinear in the generalised quadrangle. The collinearity graph is strongly regular, and
the ``Krein parameters" imply the inequality of D. Higman. We also have a divisibility
condition arising from the formula for the multiplicities of the eigenvalues of the collinearity
graph.

\begin{lemma}\label{lem:s+tdiv}
 In a thick generalised quadrangle of order $(s,t)$, we have
 \begin{enumerate}[(i)]
 \item (Higman's inequality): $s \le t^2$ and $t \le s^2$, and
 \item (Divisibility condition):  $s+t$ divides $st(s+1)(t+1)$.  % and $s+t\mid s^2(st+1)$. (not needed)
 \end{enumerate}
\end{lemma}

We will often refer to \cite{Payne:2009ly} for well-known results such as the above. 

A \textit{proper subquadrangle} of a generalised quadrangle $\Q$ is a generalised quadrangle $\Q' \not= \Q$ whose points and lines are subsets of the points and lines of $\Q$ with incidence being the restricted incidence relation from $\Q$.

\begin{lemma}[{\cite[2.2.2]{Payne:2009ly}}]\label{GQ-book-2}
Let $\Q$ be a generalised quadrangle of order $(s,t)$; let $\Q'$ be a proper subquadrangle of $\Q$
of order $(s,t')$, and let $\Q''$ be a proper subquadrangle of $Q'$ of order $(s, t'')$, where
$s>1$.  Then $t''=1$, $t'=s$, and $t=s^2$.
\end{lemma}

\begin{lemma}[{\cite[2.4.1]{Payne:2009ly}}]\label{sub-GQ-theta}
Let $\theta$ be an automorphism of the generalised quadrangle $\Q$.  The substructure of the fixed
points and lines of $\theta$ must have one of the following forms:
\begin{enumerate}[(i)]
\item no lines;
\item no points;
\item all of the lines pass through one distinguished point;
\item all of the points lie on one distinguished line;
\item a grid (every point lies at the intersection of exactly two lines);
\item a dual grid (every line contains exactly two points); or
\item a subquadrangle of order $(s', t')$ with $s' \ge 2$ and $t' \ge 2$.
\end{enumerate} 
\end{lemma}

\begin{lemma}[{\cite[2.2.1]{Payne:2009ly}}]\label{sub-GQ-order}
Let $\Q$ be a generalised quadrangle of order $(s,t)$, and let $\Q'$ be a proper subquadrangle of
$\Q$ of order $(s', t')$ (so $\Q' \neq \Q$).  Then either $s'=s$, or $s't' \le s$.
\end{lemma}

The following general calculations will prove useful.

\begin{lemma}\label{useful}
Let $\Q$ be a thick generalised quadrangle of order $(s,t)$ that has $v$ points, and let $\Q'$ be a proper
subquadrangle of $\Q$ of order $(s', t')$ that has $v'$ points.  Then the following hold:
\begin{enumerate}[(i)]
\item $v < (t+1)^5$;
\item $v < (s+1)^4$;
\item $v > s^{5/2}$; and
\item if $s't' \le s$ and $s' \neq s$ then $v/v' >t$.
\end{enumerate}
\end{lemma}

\begin{proof}
Using Higman's inequality (\ref{lem:s+tdiv}(i)) we have $v=(s+1)(st+1) \le (t^2+1)(t^3+1) <(t+1)^5$; similarly,
$v=(s+1)(st+1) \le (s+1)(s^3+1)<(s+1)^4$.  Also $v=(s+1)(st+1) >s^2\sqrt{s}=s^{5/2}$.  For (iv), we
have $v/v'=(s+1)(st+1)/((s'+1)(s't'+1)) \ge (st+1)/s>st/s=t$, since $s't' \le s$ and $s'<s$.
\qed\end{proof}

\subsection{Group actions}
 
We will assume the reader is familiar with the basic notions of group actions such as transitivity
and primitivity. Details can be found in \cite{DM}. A \textit{subdegree} of a transitive group is
the length of an orbit under the stabiliser of a point.

Finite primitive permutation groups are characterised by the O'Nan-Scott Theorem. We will follow the
description given in \cite{praegerbcc} which splits the primitive groups into eight types. The types
are distinguished by the structure and action of the minimal normal subgroups.  Table
\ref{tab:primgroups} provides a description of each type. The information given in the third column
is sufficient to identify a primitive group of the corresponding type but is not necessarily a
complete account of the known information about such a group. The fact that $k\geq 6$ in the TW case follows from \cite[Theorem 4.7B (iv)]{DM} and requires the Classification of Finite Simple Groups.

The \emph{socle} of a group is the product of all of its minimal normal subgroups. Let $N= T^k$ for
some finite nonabelian simple group $T$ and let $\pi_i$ be the projection map from $N$ onto the
$i^{\rm th}$ direct factor of $N$.  A subgroup $M$ of $N$ is called a \emph{full strip} if $M\cong T$ and, for each $i\in\{1,\ldots,k\}$,
either $\pi_i(M)\cong T$ or $\{1\}$. The \emph{support} of a strip is the subset
$I$ of all $i\in \{1,\ldots,k\}$ such that $\pi_i(M)\cong T$. Note that a nontrivial element $g$ of
a strip with support $I$ satisfies $\pi_i(g)\ne 1$ for all $i\in I$. The \emph{length} of a strip
is the cardinality of its support and two strips are said to be \emph{disjoint} if their supports are
disjoint. A \emph{full diagonal subgroup of $N$} is a full strip of length $k$.

\begin{center}
 \begin{table}
  \begin{tabular}{c|l|p{8cm}}
  Abbreviation & O'Nan-Scott type & Description\\ \hline\hline 
  HA & Affine & The unique minimal
  normal subgroup is elementary abelian and acts regularly on $\Omega$.\\ 
  HS & Holomorph simple & Two minimal normal subgroups: each nonabelian simple and regular.\\ 
  HC & Holomorph compound & Two
  minimal normal subgroups: each isomorphic to $T^k$, $k\ge 2$ and regular.\\ 
  AS & Almost simple & The unique minimal normal subgroup is nonabelian simple.\\ 
  TW & Twisted wreath & The unique minimal normal subgroup is regular and isomorphic to $T^k$ with $k\ge 6$.\\ 
  SD & Simple diagonal  & The unique minimal normal subgroup $N$ is isomorphic to $T^k$, with $k\ge 2$ and $N_{\alpha}$
  is a full diagonal subgroup of $N$. The group $G$ acts primitively on the set of $k$ simple direct factors of $N$. \\ 
  CD & Compound diagonal & The unique minimal normal subgroup $N$ is isomorphic to $T^k$ and 
  $N_{\alpha}\cong T^{\ell}$ is a direct product of $\ell$ pairwise disjoint full strips of length $k/\ell$ with $\ell,k/\ell\ge 2$. 
  The set of supports of the full strips forms a system of imprimitivity for the action of $G$ on the set of $k$ simple direct factors of $N$.\\ 
  PA & Product action & The unique minimal normal subgroup $N$ is isomorphic to $T^k$, with $k\ge 2$
  and $N_{\alpha}\cong R^k$ for some proper nontrivial subgroup $R$ of $T$.\\ \hline
  \end{tabular}
\caption{Types of primitive groups $G$ on a set $\Omega$. Here $T$ is a nonabelian simple group and $\alpha\in\Omega$.}
\label{tab:primgroups}
 \end{table}
\end{center}

In some of the O'Nan-Scott types, the distinguishing features of the type are purely group theoretic
and do not require properties of their actions. This allows us to prove the following lemma about
groups with two different primitive actions.

\begin{lemma}
\label{lem:2actions}
Let $G$ be a group with faithful primitive actions on the sets $\Omega_1$ and $\Omega_2$. Then the
possible O'Nan-Scott types of these actions are given by Table $\ref{tab:2primactions}$.
\end{lemma}

\begin{table}[H]
\begin{center}
\begin{tabular}{c | l }
\hline
Primitive type on $\Omega_1$ & Primitive type on $\Omega_2$ \\
\hline
HA & HA\\
HS &HS \\
HC & HC \\
AS & AS \\
TW & TW, SD, CD PA\\ 
SD & TW, SD, PA\\
CD & TW, CD, PA\\
PA & TW, SD, CD, PA\\
\hline
\end{tabular}
\caption{Two primitive actions}
\label{tab:2primactions}
\end{center}
\end{table}

\begin{proof}
Suppose that $G$ acts primitively on $\Omega_1$ with O'Nan-Scott type HA, HS, HC or AS. Then as each
of these types has a group theoretic structure that uniquely determines its primitive type, the action 
on $\Omega_2$ must be of the same type. A primitive group of each of the remaining four types has a unique minimal
normal subgroup $N$, which is isomorphic to $T^k$ where $T$ is a nonabelian simple group $T$ and $k\ge
2$. If $G$ is of type SD then it acts primitively on the set of $k$ simple direct factors of $N$,
while a group of type CD acts imprimitively on this set. Thus $G$ cannot have both a primitive
SD action and a primitive CD action. This completes the last four rows of Table
\ref{tab:2primactions}.  \qed\end{proof}

\section{Some basic deductions from a primitive action on a generalised quadrangle}\label{section:asorpa}

We begin our analysis by deducing what follows from the simple premise of a group acting on the
points of a generalised quadrangle. Some of these results follow immediately from elementary number
theoretic considerations.

\begin{lemma}\label{s=t}
If $\Q$ is a thick generalised quadrangle of order $(s,t)$ with $s=t$ and the number of points is equal
to $\delta^k$ with $k\geq 2$, then $s=7$, $\delta=20$ and $k=2$.
\end{lemma}
\begin{proof}
Since $s=t$, the number of points is $v=(s+1)(st+1)=(s+1)(s^2+1)=(s^4-1)/(s-1)$.  
By the results of Nagell and Ljunggren (see \cite{Ljunggren:1943qf}), if $v=\delta^k$ with $k\geq 2$ then $s=7$, $\delta=20$ and $k=2$.  \qed\end{proof}

\begin{lemma}\label{not-coprime}
Let $(s,t)$ be the parameters of a thick generalised quadrangle.  If $s+1$ divides $t+1$, or $t+1$ divides
$s+1$, then $s$ and $t$ are not coprime.
\end{lemma}

\begin{proof}
Suppose that $s$ and $t$ are coprime and, by taking duals if necessary, that $s+1$ divides $t+1$.  
First note that $s+t$ and $st$ are also coprime and so by Lemma \ref{lem:s+tdiv}, $s+t$ divides $st+1$.  Now
$$st+1=s(s+t)-s^2+1,$$ therefore, $s+t$ divides $(s-1)(s+1)$. Since
$s+1$ divides $t+1$, the latter implies that $s+t$ divides $(s-1)(t+1)$.  So $s+t$
divides $(st+1)-(s-1)(t+1)=2+t-s=2(t+1)-(s+t)$, and hence, $s+t$ divides $2(t+1)$.
Now $s>1$ and so $s+t>t+1$. So $s+t=2(t+1)$, and thus $s=t+2$.  This is a contradiction (as $s+1$
would divide $s-1$).  \qed\end{proof}

\begin{theorem}[Benson's Theorem \cite{Benson:1970ve}]\label{Benson}
Let $\mathcal{Q}$ be a generalised quadrangle of order $(s,t)$ and let $\theta$ be an automorphism
of $\mathcal{Q}$. Let $f$ be the number of fixed points of $\theta$ and let $g$ be the number of
points $x$ for which $x^\theta$ is collinear with $x$. Then
$$(1+t)f +g\equiv st+1 \pmod{s+t}.$$
\end{theorem}

The following consequence of Benson's Theorem is an analogue of Corollaries 5.3 and 6.3 in the paper of 
Temmermans, Thas and Van Maldeghem \cite{Temmermans:2009fk}, for generalised quadrangles.

\begin{lemma}\label{line-fixed}
Let $\theta$ be a fixed-point-free automorphism of a thick generalised quadrangle $\mathcal{Q}$ of order
$(s,t)$. Suppose $\theta$ has order $2$ or $3$.  If $s$ and $t$ are not coprime, then $\theta$ fixes
a line of $\mathcal{Q}$.
\end{lemma}

\begin{proof}
Let $g$ denote the number of points collinear with their image under $\theta$. Then by Benson's Theorem \ref{Benson}, $g\equiv st+1 \pmod{s+t}$. As $s$, $t$ are not coprime there is an integer $m>1$ dividing both $s$ and $t$.  If $s+t$ divides $st+1$ then $m$ would divide both $st$ and $st+1$ which is not possible. Therefore $s+t$ does not divide $st+1$ and so $g$ is non-zero. Therefore, there is at least one point $x$ where $x^\theta$ is collinear with $x$. If $\theta$ has order $2$, then $\theta$ fixes the line joining $x$ and $x^\theta$. If $\theta$ has
order $3$, then $x, x^\theta, x^{\theta^2}$ are pairwise collinear, and as $\mathcal{Q}$ contains no triangles, the  points $x, x^\theta, x^{\theta^2}$ lie on a common line that is fixed by $\theta$.
\qed\end{proof}

\begin{lemma}\label{lemma:onlyASPAleft}
Let $\Q$ be a thick generalised quadrangle of order $(s,t)$, and suppose $G$ is a group of
automorphisms of $\Q$ acting primitively on the points of $\Q$ and primitively on the lines of
$\Q$. Then either:
\begin{enumerate}[(i)]
 \item one of these two primitive actions is of product action type; or
 \item $G$ is almost simple.
\end{enumerate}
\end{lemma}

\begin{proof}
Suppose first that $G$ is a primitive group of type HA in its action on the points of $\Q$. Then by Lemma \ref{lem:2actions}, $G$ is
also primitive of type HA in its action on lines and so the unique minimal normal subgroup $N$ of $G$ acts
regularly on both points and lines. Thus the number of lines is equal to the number of points and so
$s=t$. The number of points is equal to $|N|=p^d$ for some prime $p$. Since the number of points is $(s+1)(st+1)$, 
it is not a prime and so $d\geq 2$.  Lemma \ref{s=t} implies
that this case cannot occur.

Next suppose that $G$ is primitive of type HS or HC on points. By Lemma \ref{lem:2actions}, $G$ is
primitive of the same type on lines. Moreover, a minimal normal subgroup $N$ of $G$ is regular on
both points and lines and so $s=t$. Since $N\cong T^k$, where $T$ is a nonabelian simple group, it follows from the
Feit-Thompson Theorem~\cite{FeitThompson} that $N$ contains an involution $g$. However, as $g$ fixes no points or lines,
this contradicts Lemma \ref{line-fixed} and so this case cannot occur.

Next suppose that $G$ is primitive of type TW on points and let $N= T^k$ be the unique minimal
normal subgroup of $G$, where $T$ is a nonabelian simple group. Then $N$ acts regularly on points. Since $G$ acts primitively on lines, $N$ must act transitively on lines and so the number of lines divides $|N|$, which is the number of points. Hence
$t+1$ divides $s+1$, and so by Lemma \ref{not-coprime}, $s$ and $t$ are not coprime.  By the
Feit-Thompson Theorem, $T$ contains an involution $g$. Then $g'=(g,1\ldots,1)$ is an element of $N$ and therefore fixed-point-free. Hence by Lemma \ref{line-fixed}, $g'$ fixes a line. By Table \ref{tab:primgroups}, in the TW case the stabiliser in $N$ of a line is trivial, in the SD case it is a full diagonal subgroup, and in the CD case, it is a product of $\ell$ pairwise disjoint strips of length $k/\ell\ge 2$. Since none of these can contain $g'$, it follows from Lemma \ref{lem:2actions} that $G$ must be of type PA on lines.

Next suppose that $G$ is primitive of type SD or CD on points and let $N\cong T^k$ be the unique
minimal normal subgroup of $G$, where $T$ is a nonabelian simple group. Then the number of points is equal to $|T|^{k-\ell}$ for some
$\ell\ge 1$.  By Lemma \ref{lem:2actions}, the action of $G$ on lines is of type TW, SD, CD or
PA. Suppose that the action on lines is not of type PA. Then the number of lines is $|T|^i$ for some
integer $i$. Hence either the number of points divides the number of lines or vice versa. Thus by
Lemma \ref{not-coprime}, $s$ and $t$ are not coprime. By the Feit-Thompson Theorem, $T$ contains an involution $g$ and so the element $g'=(g,1\ldots,1)$ lies in $N$.  By Table \ref{tab:primgroups}, the stabiliser in $N$ of a point is either a full diagonal subgroup (the SD case)  or a product of $\ell$ pairwise disjoint strips of length $k/\ell\ge 2$ (the CD case). Thus $g'$ fixes no points. Also the stabiliser in $N$ of a line is either trivial (the TW case) or a full diagonal subgroup (the SD case)  or a product of $\ell$ pairwise disjoint strips of length $k/\ell\ge 2$ (the CD case). Thus $g'$ also fixes no lines, contradicting Lemma \ref{line-fixed}. Hence
$G$ must be of type PA on lines. Thus by Lemma \ref{lem:2actions}, the possible combinations of the
action of $G$ on points and lines is given by the uncrossed entries of Table \ref{tab:poss} and so the
lemma holds.\qed\end{proof}

\begin{table}
\begin{center}
\begin{tabular}{c | l }
\hline
Primitive type on points & Primitive type on lines \\
\hline
\xout{HA} & \xout{HA}\\
\xout{HS} & \xout{HS} \\
\xout{HC} & \xout{HC} \\
AS & AS \\
TW & \xout{TW}, \xout{SD}, \xout{CD}, PA\\ 
SD & \xout{TW}, \xout{SD}, PA\\
CD & \xout{TW}, \xout{CD}, PA\\
PA & TW, SD, CD, PA\\
\hline
\end{tabular}
\caption{Possible actions on points and lines}\
\label{tab:poss}
\end{center}
\end{table}

\section{Product Action type} \label{section:productaction}

In this section we consider the case where
one of the primitive actions is of PA type; by duality we can assume that this is the action on points.
The argument proceeds by a series of lemmas that all use common hypotheses and notation, and so we first establish this common
setting.

\begin{hyp}\label{pahype}
Throughout this section, we assume that $\Q$ is a thick generalised quadrangle of order $(s,t)$ with a group $G$ of automorphisms
whose action on the points of $\Q$ is primitive of type PA. The group $G$ has a unique minimal normal subgroup $N=T^k$ where $k > 1$ and
$T$ is a nonabelian simple group. From the properties of the PA type, the set of points can be identified with 
$\Delta^k$ for some set $\Delta$ with $|\Delta|=\delta$. Furthermore, the group $G$ is a subgroup of $H\Wr S_k$, where $H$ is a primitive group of type AS on $\Delta$ with minimal normal subgroup isomorphic to $T$.
\end{hyp}

We first deal with the case $s=t$.

\begin{lemma}\label{lem:PAs=t} 
Under Hypothesis~\ref{pahype}, the generalised quadrangle $\Q$ does not have order $(s,s)$.
 \end{lemma}

\begin{proof}
For a contradiction, suppose that $\Q$ does have order $(s,s)$.  It then follows from Lemma \ref{s=t} that $s=7$, $\delta=20$ and $k=2$. 
Moreover, $G$ is a subgroup of the wreath product $H\Wr S_2$ endowed with its natural product action on $\Delta^2$, and
$H$ is a primitive group on $\Delta$.  There are only four primitive groups of degree $20$:
$\mathsf{PSL}(2,19)$, $\mathsf{PGL}(2,19)$ (acting on the $20$ points of the projective line),
$A_{20}$ and $S_{20}$. Thus $H$ is one of these groups.

The socle of
$G$ is $N=T_1\times T_2$ with $T_1\cong T_2\cong T$. By direct inspection, we see that $T$ is either
$3$-transitive (if $T=A_{20}$) or the stabiliser $T_{x_1,x_2}$ of two distinct points
$x_1,x_2$ of $\Delta$ has two orbits of length $9$ on
$\Delta\setminus\{x_1,x_2\}$ (if $T=\mathsf{PSL}(2,19)$).

Let $\alpha=(x_1,x_2)$ and $\beta=(x_1',x_2')$ be two collinear points and let $\ell$ be
the line through $\alpha$ and $\beta$. Now $N_{\alpha,\beta}=(T_1)_{x_1,x_1'}\times
(T_2)_{x_2,x_2'}$ fixes the line $\ell$, which has $8$ points. In particular, $\ell$ is a
union of $N_{\alpha,\beta}$-orbits. Since each non-trivial orbit of $(T_1)_{x_1,x_1'}$ has length at
least $9$, we obtain that every $N_{\alpha,\beta}$-orbit has length at least $9$, a contradiction.
\qed\end{proof}

By Lemma \ref{lem:2actions}, if $G$ is primitive of type PA on points and primitive on lines
then the action on lines is of type TW, SD, CD or PA. Thus in each of these cases, we can write the line set as
$\Gamma^{k'}$ for some set $\Gamma$.  We let $\gamma=|\Gamma|$.

%If the group action on the line set is of TW type, we have $k'=k \ge 6$, and $\gamma=|T|$ for some
%simple group $T$.  Furthermore in this case, $s+1$ divides $t+1$, so by Lemma~\ref{not-coprime},
%$\gcd(s,t)>1$.
%
%If the group action on the line set is of SD type, we have $k'=k-1$, and $\gamma=|T|$ for some
%simple group $T$. If the group action on the line set is of CD type, we have $k'=k-a$, where $k=ab$
%for some $a, b >1$, and $\gamma=|T|$ for some simple group $T$. If the group action on the line set
%is of PA type, we have $k'=k$.

\begin{lemma}\label{support-k}
In addition to Hypothesis~\ref{pahype}, suppose that $G$ acts primitively of type TW, SD, CD,
or PA on lines and that if the action on lines is of type PA, then there are more lines than points. 
Then any two points in a line, viewed as elements of $\Delta^k$, are at Hamming distance $k$.
In addition, if the action on lines is of type TW, SD, or CD and $\alpha$ and $\beta$ are
collinear points, then the subgroup of $T$ that fixes both the $i^{\rm th}$ entry of $\alpha$ and
the $i^{\rm th}$ entry of $\beta$ is the identity.
\end{lemma}

\begin{proof}
Fix a point $\alpha$ which, by point-transitivity, we can assume is  given by $\alpha=(x,\ldots, x)$ for some $x \in
\Delta$.  Let $R = T_x$ be the stabiliser of $x$ in the action of $T$ on the points of $\Delta$
(which is nontrivial under PA action).  Let $\beta$ be any point that is collinear with $\alpha$,
and let $\ell$ be the line through $\alpha$ and $\beta$.  Let $R_i$ be the subgroup of $N$ with $\pi_j(R_i)=\{1\}$ for all $j\neq i$ and $\pi_i(R_i)=R$.

If the action on the lines is of TW, SD, or CD type, then each $R_i$ is semiregular in its action on
lines, since the stabiliser of a line in $N=T^k$ is either trivial (the TW case) or the
product of full strips all of length at least 2.  In particular, $R_i$ does not fix $\ell$,
so cannot fix $\beta$.  Since the action of $R_i$ affects only the $i^{\rm th}$ coordinate of
$\beta$, and $R$ fixes $x$, the $i^{\rm th}$ entry of $\beta$ cannot be $x$.  So
the Hamming distance between $\beta$ and $\alpha$ is $k$, as claimed.

Furthermore, since $R_i$ is semiregular in its action on lines, no two elements of $R_i$ can take
$\beta$ to the same point, so if $y$ is the $i^{\rm th}$ entry of $\beta$, we have $|y^R|=|R|$.
Hence $T_{x,y}=R_y=\{1\}$, as claimed.

Suppose now that the action on the lines is of PA type.  Then the line $\ell$ is in
$\Gamma^{k'}$. Let $S_i$ ($1 \le i \le k'$) be the stabiliser of the $i^{\rm th}$ entry of $\ell$ in
the action of $T$ on the elements of $\Gamma$.  Then all of the $S_i$ have the same order.  As
both of the actions are of PA type there are more lines than points, so $\gamma>\delta$, giving $|T|/\delta >|T|/\gamma$; thus $|R|>|S_i|$.  Hence for any $i$, there is some $r_i \in R_i$ such that $\ell^{r_i} \neq \ell$.  Again,
this means that $r_i$ cannot fix $\beta$; the action of $r_i$ affects only the $i^{\rm th}$ entry of
$\beta$ and fixes $x$, so the $i^{\rm th}$ entry of $\beta$ cannot be $x$.  So again the Hamming
distance between $\beta$ and $\alpha$ is $k$.  \qed\end{proof}

\begin{cor}\label{s+1-le-delta}
In addition to Hypothesis~\ref{pahype}, suppose that $G$ acts primitively of type TW, SD, CD,
or PA on the lines of $\Q$ and that if the action on lines is of type PA,  then $\Q$ has more lines than points.
Then $s+1 \le \delta$.
\end{cor}

\begin{proof}
By Lemma~\ref{support-k}, two points on a line $\ell$ have different elements of $\Delta$ in any fixed coordinate position. 
Therefore there are at most $\delta$ values available for any fixed coordinate position, and so at most $\delta$ points on $\ell$.
\qed\end{proof}

\begin{cor}\label{k-le-3}
In addition to Hypothesis~\ref{pahype}, suppose that $G$ acts primitively of type TW, SD, CD,
or PA on the lines of $\Q$ and that if the action on lines is of type PA, then $\Q$ has more lines than points. Then
$k \le 3$.
\end{cor}

\begin{proof}
By Corollary~\ref{s+1-le-delta}, we have $s+1 \le \delta$.  We also have the number of points is
$\delta^k=(s+1)(st+1)$. So $st+1 \ge \delta^{k-1}$.  Hence $t \ge (\delta^{k-1}-1)/s \ge
(\delta^{k-1}-1)/(\delta-1)$.  From Higman's inequality (\ref{lem:s+tdiv}(i)), we also have $s^2 \ge t$, so $s^2 \ge \delta^{k-2}$, so $s \ge
\delta^{(k-2)/2}$.  Comparing this with $s <\delta$ forces $(k-2)/2<1$ so $k<4$, as claimed.
\qed\end{proof}

\begin{cor}\label{TW-and-CD}
Under Hypothesis~\ref{pahype}, the action of $G$ on the lines of $\Q$ is not primitive of type TW or CD.
\end{cor}

\begin{proof}
In a group of TW type, we have $k \ge 6$ (note that this requires the Classification of Finite Simple
Groups).  In a group of CD type, $k$ is the product of $\ell$, $k/\ell \ge 2$ and so $k \ge 4$.  These contradict Corollary~\ref{k-le-3}.
\qed\end{proof}

\begin{lemma}\label{SD}
Under Hypothesis~\ref{pahype}, the action of $G$ on the lines of $\Q$ is not primitive of type SD.
\end{lemma}

\begin{proof}
By the properties of primitive groups of type SD, the line set can be written as $\Gamma^{k-1}$ where $\Gamma$ can be identified with $T$. Moreover, a subgroup $T^{k-1}$ of  the socle $N=T^k$ of $G$ acts regularly on the set of lines.  Let $\alpha$ be a point incident with a line $\ell$.   Then the intersection $R^{k-1}$ of  $T^{k-1}$ with $N_{\alpha}$ acts semiregularly on lines through $\alpha$. Hence no two elements of $R^{k-1}$
can take $\ell$ to the same line, so $|\ell^{R^{k-1}}|=|R^{k-1}|$.  Hence $|R^{k-1}|$ divides the
number of lines through $\alpha$, but the number of lines through $\alpha$ is $t+1$.  So we have
$t+1 \ge |R^{k-1}|$.  Now $|R|=|T|/\delta$, and so $t+1 \ge (|T|/\delta)^{k-1}$.

Now, $(t+1)(st+1)=|T|^{k-1}$.  Combining this with $t+1 \ge (|T|/\delta)^{k-1}$ gives $st+1 \le
\delta^{k-1}$.  Recall that $(s+1)(st+1)=\delta^k$, and so $s+1 \ge \delta$.  Thus by
Corollary~\ref{s+1-le-delta}, we have $s+1=\delta$.

Now since the line $\ell$ contains $\delta$ points, pairwise at Hamming distance $k$ from one
another (by Lemma~\ref{support-k}), every element of $\Delta$ must appear as the first entry of some
point on $\ell$.  Now the second part of the statement of Lemma~\ref{support-k} tells us that
$T_{x,x'}=1$ for every pair $x\neq x' \in \Delta$.  Furthermore $T$ cannot be regular in its action (a
property of PA groups).  So by definition, $T$ is a Frobenius group.  However, the Structure Theorem of
Frobenius groups (see \cite[p 86]{DM}), implies that no finite simple group is a Frobenius group.
\qed\end{proof}

We can therefore assume that the action on both points and lines is of PA type, and that $k \le 3$.

\begin{lemma}\label{line-stabiliser}
In addition to Hypothesis~\ref{pahype}, assume that $G$ acts primitively of type PA on the lines of $\Q$ and that $\Q$ has
more lines than points. Then the stabiliser $N_{\ell}$ of any line $\ell$ fixes every point in that line. Furthermore $s+1$ divides $t+1$.
\end{lemma}

\begin{proof}
Since $\ell \in \Gamma^{k'}$ and $k'=k$ when we have two actions of PA type of the same group, we have $\ell=(y_1,\ldots, y_k)$ for some $y_1, \ldots, y_k \in \Gamma$.  Let $\overline{S_i}=T_{y_i}$ for $1 \le i \le k$, and let $S_i$ be the subgroup of $N$ with $\pi_j(S_i)=\{1\}$ for all $j\neq i$ and $\pi_i(S_i)=\overline{S_i}$.  Let $\alpha$ be any point of $\ell$.  Then since $S_i$ fixes $\ell$, we must
have $\alpha^{s_i} \in \ell$ for any $s_i \in S_i$, and $s_i$ changes only the $i^{\rm th}$ entry of $\alpha$.  By Lemma~\ref{support-k}, no other point in $\ell$ has any entries in common with $\alpha$, certainly not $k-1$ entries in common, so it must be the
case that for every $s_i \in S_i$, $\alpha^{s_i}=\alpha$.  Thus, for every $i$, we have shown that $S_i$ fixes
$\alpha$.  This means that $N_{\ell}=\overline{S_1} \times \ldots \times \overline{S_k}$ also fixes $\alpha$.  Since
$\alpha$ is an arbitrary point of $\ell$, this completes the proof of the first statement.

In fact, this shows that $\overline{S_1} < R_1$ where $R_1$ is the subgroup of $T$ that fixes the first
coordinate of $\alpha$.  So $|\overline{S_1}|$ divides $|R_1|$, meaning $\delta \mid \gamma$ (since $\delta =
|T|/|R_1|$ and $\gamma=|T|/|\overline{S_1}|$).  Then we have $(t+1)/(s+1)=(\gamma/\delta)^k$, and so $s+1$ divides $t+1$.  \qed\end{proof}

The following result yields a generalised subquadrangle of the generalised quadrangle $\mathcal{Q}$.

\begin{lemma}\label{GQ-book-1}
In addition to Hypothesis~\ref{pahype}, assume that $G$ acts primitively of type PA on the lines of $\Q$ and that $\Q$ has
more lines than points. Let $\ell$ be an arbitrary line of $\Q$, and let $S$ be any subgroup
of $N_{\ell}$ that contains $T_y\times 1^{k-1}$, where $y\in \Gamma$ is the first entry of $\ell$.  Then the
incidence substructure $\Q_{S}$ consisting of the points and lines that are fixed by
every element of $S$ is described by one of the following:
\begin{enumerate}[(i)]
\item there are exactly two lines of $\Q_S$ through every point of $\Q_S$, and $\Q_S$ is an $s+1$ by
  $s+1$ grid; or
\item $\Q_S$ is a proper thick subquadrangle of order $(s, t')$ with $2 \le t'<t$.
\end{enumerate}
\end{lemma}

\begin{proof}
Since $S$ is a subgroup of $N_\ell$, we have at least one line, $\ell$, in $\Q_S$.  By
Lemma~\ref{line-stabiliser}, every point of $\ell$ is fixed by $S$, so $\Q_S$ has at least $s+1$
points.  In fact, this lemma tells us that whenever $\ell'$ is fixed by $S$, so is every point of
$\ell'$, so every line in $\Q_S$ has $s+1$ points.  Since $s\ge 2$, there are more than $2$ points
in every line of $\Q_S$. Hence $\Q_S$ is not a dual grid. Moreover, since $T_y\neq 1$ (a property of actions of type PA given in Table \ref{tab:primgroups}), $S$ acts nontrivially on the points of $\Q$ and so $\Q\neq \Q_S$. Thus by Lemma~\ref{sub-GQ-theta}, it remains to show that all points do not lie on a distinguished line and all lines do not pass through a distinguished point.

Suppose that $\alpha_1=(x_{1,1}, \ldots, x_{1,k})$, \ldots\ , $\alpha_{s+1}=(x_{s+1,1}, \ldots, x_{s+1,k})$
are the points of $\ell$.  Then by Lemma~\ref{support-k} $x_{i,j} \neq x_{i',j}$ for any $j$ and any
$i \neq i'$, and every point of the form $(x_{i_1,1},\ldots, x_{i_{k},k})$ where $i_1, \ldots, i_{k}
\in \{1, \ldots, s+1\}$ is fixed by $N_{\ell}$ and hence by $S$, so there are at least $(s+1)^{k}$
points in $\Q_S$.  In particular, the points of $\Q_S$ do not lie on a distinguished line. Thus let $\alpha$ be a point of $\Q_S$ not on $\ell$ and let $\ell'$ be the unique line of $\Q$ passing through $\alpha$ and meeting $\ell$. Let $\beta$ be the point of intersection of $\ell$ and $\ell'$. Since $S$ fixes all points on $\ell$, it fixes $\beta$ and hence $S$ also fixes $\ell'$. Thus $\ell'\in\Q_S$. Since $(s+1)^k>s+1+s$, we can find a third point $\alpha'$ fixed by $S$ and on neither $\ell$ nor $\ell'$. Since the points of $\ell$ and $\ell'$ differ from $\beta$ in all coordinates (Lemma \ref{support-k}), $\alpha'$ can be chosen to have the first entry in common with $\beta$. Let $\ell''$ be the unique line of $\Q$ through $\alpha'$ and meeting $\ell$, and let $\beta'$ be the point of intersection. By Lemma \ref{support-k}, $\beta\neq \beta'$. Since $S$ fixes $\beta'$ and $\alpha''$, it also fixes $\ell''$ and so $\ell''\in\Q_S$. Thus $\ell,\ell',\ell''$ are three lines of $\Q_S$ that do not pass through a common point. Thus not all lines of $\Q_S$ pass through a distinguished point and the result follows.
 \qed\end{proof}

\begin{lemma}\label{no-PA}
Under Hypothesis~\ref{pahype}, the group $G$ does not act primitively of type PA on the lines of $\Q$.
%The situation of a generalised quadrangle $\Q$ having a group of PA type acting on both the points
%and the lines, with more lines than points cannot occur.
\end{lemma}

\begin{proof}
By Lemma~\ref{lem:PAs=t} and taking duals if necessary, we may suppose that $\Q$ has more lines than points.
Let $y\in\Gamma$, and let $\ell$ be the line $(y,\ldots, y)$.  Consider the substructure $\Q'$ of fixed points and lines of $T_{y}^{k-1}
\times 1$  and the substructure $\Q''$ of fixed
points and lines of $T_y^k$.  Both groups are subgroups of $N_\ell$ that
contain $T_y \times 1^{k-1}$, so by Lemma~\ref{GQ-book-1}, the structures $\Q'$ and $\Q''$ are each
either an $s+1$ by $s+1$ grid, or a proper thick subquadrangle.
Now, $T_y \neq 1$ because of the properties of primitive groups of type PA,  so the set of points of $\Q''$ is a proper subset of
the set of points of $\Q'$, which is a proper subset of the set of points of $\Q$.  So $t>t'>t''\ge
1$, and hence $t' \ge 2$.  Therefore $\Q'$ is a proper subquadrangle of order $(s,t')$ where $t' \ge 2$, 
and $\Q''$ is a proper subquadrangle of order $(s,t'')$ where $t'' \ge 1$ (where if $t''=1$, $\Q''$ is a grid).  Now we can apply Lemma~\ref{GQ-book-2} to obtain $t=s^2$.  By
Lemma~\ref{line-stabiliser}, we have $s+1$ divides $t+1 = s^2+1=s(s+1)+1-s$, so $s+1$ divides $s-1$, a
contradiction since $s>1$.  \qed\end{proof}

We can now prove Theorem~\ref{main1}.

\begin{proof}
By Lemma~\ref{lemma:onlyASPAleft}, if $G$ is not almost simple, then the action of $G$ is of PA type on either the points or the lines of the generalised quadrangle.  By duality, we can assume that the action of $G$ is of PA type on the points.  
Lemma~\ref{lem:2actions} shows that the action of $G$ on the lines must be of TW, CD, SD, or PA type on the lines. The first three of these possibilities are ruled out by Corollary~\ref{TW-and-CD} and Lemma~\ref{SD} and Lemma~\ref{no-PA} rules out the PA case, thereby completing the proof.
%If the action on both points and lines is of PA type then we have not yet required duality, so we can use duality to assume without loss of generality that $t \ge s$; that is, there are at least as many lines as points.
%Lemma~\ref{lem:PAs=t} tells us that the situation $s=t$ cannot arise, and Lemma~\ref{no-PA} completes the proof.
\qed\end{proof}

\section{The Alternating Groups}\label{section:alternating}

Our goal in this section is to prove Theorem \ref{theorem:alternatinggroups} when the socle of $G$ is an alternating group.
Note that the automorphism group of the unique
generalised quadrangle of order $(2,2)$ is isomorphic to $S_6$.  An almost simple group with socle
$A_n$ must be $A_n$ or $S_n$, or $n=6$ and it is one of $A_6$, $S_6$, $M_{10}$, $\mathsf{PGL}(2,9)$ or
$\mathsf{P\Gamma L}(2,9)$.  First we deal with these exceptional cases.

\begin{lemma}\label{lemma:strange6}
The groups $M_{10}$, $\mathsf{PGL}(2,9)$ and $\mathsf{P\Gamma L}(2,9)$ do not act primitively on the
points of a thick generalised quadrangle.
\end{lemma}

\begin{proof}
Each group has exactly three maximal subgroups with index greater than $2$, and their indices are precisely
$45$, $36$ and $10$.  The only one of these three of the form $(s+1)(st+1)$ is $45$, whereby
$s=4$ and $t=2$. By \cite[5.3.2(ii)]{Payne:2009ly}, there is a unique generalised quadrangle of
order $(4,2)$ and its full automorphism group is $\mathsf{P\Gamma U}(4,2)$.  It is not difficult to
check that $M_{10}$, $\mathsf{PGL}(2,9)$ and $\mathsf{P\Gamma L}(2,9)$ are not subgroups of
$\mathsf{P\Gamma U}(4,2)$.  \qed\end{proof}

We can adapt the following two results from similar statements proven in~\cite{Schneider:2008zr}:

\begin{lemma}\label{lemma:power6}
Let $\Q$ be a thick generalised quadrangle of order $(s,t)$ and let $G$ be a point-transitive group of automorphisms of $\Q$. If $\alpha$ is a point of $\Q$, then $|G|/|G_\alpha| <
(1+t)^5$. If $G$ acts transitively on flags, then $|G| < |G_\alpha|^6$.
\end{lemma}

\begin{proof}
Since $G$ acts transitively, the Orbit-Stabiliser Theorem tells us that the number of points of $\Q$ equals $|G|/|G_{\alpha}|$. Lemma~\ref{useful}(i) gives the first inequality.  If $G$ acts flag-transitively, then $G_\alpha$ acts transitively on the $t+1$ lines through $\alpha$, so the Orbit-Stabiliser Theorem implies $|G_\alpha| \ge 1+t$ and the result follows.
\qed\end{proof}

\begin{lemma}\label{lemma:alternatinggroups}
Let $G=A_{n}$ or $S_n$ act flag-transitively on a thick generalised quadrangle $\Q$ with point stabiliser
$H$. Then one of the following holds:
\begin{enumerate}[(i)]
\item $H$ is intransitive in its action on $\{1, \ldots, n\}$;
\item $H$ is imprimitive in its action on $\{1, \ldots, n\}$; or
\item $n \le 47$.
\end{enumerate}
\end{lemma}

\begin{proof}
The group $H$ is intransitive, imprimitive or primitive on $\{1,\ldots,n\}$. If $H$ is primitive,
then a result of Mar\'oti \cite{Maroti:2002fk} implies that one of the following holds:
\begin{itemize}
\item $H$ is  a Mathieu group $M_n$ with $n=11,12,23,24$;
\item $H\le S_m\Wr S_k$ with $n=m^k$, $m\ge 5$ and $k\ge 2$; or 
\item $|H|\le n^{1+\lfloor \log_2(n)\rfloor}$.
\end{itemize}

Suppose $H\le S_m\Wr S_k$, where $n=m^k$, $m\ge 5$ and $k\ge 2$.  Then $|H|^6$ divides 
$(m!)^{6k}(k!)^6$. By Lemma \ref{lemma:power6}, we have $|H|^6 > |G|$ and so there exists a prime
$p$ dividing $|H|$ with $p\le \max\{m,k\}$ such that $|H|^6_p \ge |G|_p$, where for an integer $r$,
$r_p$ denotes the largest power of $p$ that divides $r$. Therefore,
$$\left((m!)^{6k}(k!)^6\right)_p\ge \left(\tfrac{1}{2} (m^k)!\right)_p.$$ By the calculations in the
proof of \cite[Lemma 5.3]{Schneider:2008zr}, $\left((m!)^{6k}(k!)^6\right)_p\le \frac{12mk+12k}{p}$,
and since $p^2\le m^k$, we have $\left(\tfrac{1}{2}(m^k)!\right)_p\ge m^k/p$.  So $12mk+12k\ge m^k$
and hence $5\le m\le 25$ and $k=2$, or $5\le m\le 6$ and $k=3$.  However, the only value for which
$(m!)^{6k}(k!)^6\ge \tfrac{1}{2} (m^k)!$ is $(m,k)=(5,2)$, that is, $n=25$.

If $|H|\le n^{1+\lfloor \log_2(n)\rfloor}$ then by \cite[Lemma 5.2]{Schneider:2008zr}, we have
$|H|^6<|H|^{12}<|G|$ when $n\ge 107.$ Calculations show that we actually have $|H|^6<|G|$ for all $n\ge
48$. Hence by Lemma \ref{lemma:power6} it follows that $n\le 47$.  \qed\end{proof}

The remaining three subsections in this section consider the cases where the action of $H$ on $\{1, 2, \ldots, n\}$ is intransitive, imprimitive or primitive with $n \le 47$. Together, they complete the proof of Theorem \ref{theorem:alternatinggroups} in the case where $\soc(G)$ is an alternating group.

%%%%%%%%%%%%%%%%%%%%%%%%%
%
%    Intransitive case
%
%%%%%%%%%%%%%%%%%%%%%%%%%

\subsection{Intransitive stabiliser}

In this subsection we deal with the situation summarised in the following hypothesis, which will be used in the statements of most of our results.

\begin{hyp}\label{hyp-intrans}
Suppose that $G$ is $A_n$ or $S_n$ with $n\geq 5$, acting primitively on the
set of points of a thick generalised quadrangle $\Q$ of order $(s,t)$.  Further suppose that the stabiliser of a point is intransitive in its action on
$\{1,\ldots,n\}$, so that we can identify the points of $\Q$ with the subsets of cardinality
$k$ from a set of cardinality $n$. 
\end{hyp}

The goal of this subsection is to show that this situation cannot occur unless $n=6$.

\begin{lemma}\label{share-i-collinear}
Under Hypothesis~\ref{hyp-intrans}, let $\alpha$ and $\beta$ be collinear points of $\Q$ such that $|\alpha \cap
\beta|=i$. Then whenever $|\alpha' \cap \beta'|=i$, we have that $\alpha'$ and $\beta'$ are
collinear in $\Q$.
\end{lemma}

\begin{proof}
We show this when $\alpha'=\alpha$; since $G$ acts transitively on the points, the result will follow.  
Clearly, there is some $g \in S_n$ such that $(\alpha\cap \beta)^g=\alpha \cap \beta'$, $(\beta\setminus\alpha)^g=\beta' \setminus \alpha$, and 
$(\alpha\setminus \beta)^g=\alpha \setminus \beta'$.  Thus $g$ fixes $\alpha$ and takes
$\beta$ to $\beta'$, so $\alpha$ and $\beta'$ must be collinear.  If $g \in A_n$, then we are done.

If $i \ge 2$ or $k-i \ge 2$ or $n-3k+2i+i'-j \ge 2$ where $i'=|\beta \cap \beta'|$ and $j=|\alpha \cap \beta \cap
\beta'|$, then we can add another 2-cycle to $g$ if necessary, to get an element of $A_n$.  
Since $A_n$ is nonabelian simple we have $n \ge 5$, so these conditions are not satisfied only when $n=5$, $k=2$,
$i=1$, and $i'=j \in \{0,1\}$.  Then without loss of generality $\alpha=\{1,2\}$, $\beta=\{2,3\}$, and
$\beta'$ is either $\{1,4\}$ or $\{2,4\}$.  In the first case, $g=(1 \ 2)(3\ 4) \in A_5$ and $g$ fixes $\alpha$
and takes $\beta$ to $\beta'$; in the second case, $g=(3\ 4\ 5)$ plays this role.  \qed\end{proof}

\begin{lemma}\label{k-1-non-collinear-pt}
Under Hypothesis~\ref{hyp-intrans}, if $\alpha$ and $\beta$ are collinear points of $\Q$ with
$|\alpha \cap \beta|=i$, then there is some point $\beta'$ such that $|\alpha \cap
\beta'|=i$ but $\beta'$ is not collinear with $\beta$.
\end{lemma}

\begin{proof}
Let $\ell$ be the line through $\alpha$ and $\beta$.  Towards a contradiction, suppose that whenever $|\alpha \cap
\beta'|=i$, $\beta'$ is collinear with $\beta$ and hence by Lemma~\ref{share-i-collinear}, $\beta'$ is on $\ell$.

Consider the generalised Johnson graph formed on the $k$-subsets of $\{1,\ldots, n\}$, with two
vertices adjacent if and only if their intersection has cardinality $i$.  Clearly, this is a
connected graph.  We prove by induction on $d$, the distance between $\alpha$ and $\alpha_1$ in this graph,
that all of the vertices in this graph must be on $\ell$.  This will be a contradiction, since not
every point of $\Q$ is on a single line.

The assumption made in the first paragraph tells us that all of the neighbours of $\alpha$ in this graph
are on $\ell$, establishing the base case of $d=1$, and since $\alpha$ itself is on $\ell$, we have the
base case of $d=0$ also.  Suppose that $\alpha_1$ is at distance $d+1$ from $\alpha$.  Then it has a neighbour
$\alpha_2$ that is at distance $d$ from $\alpha$, so by induction, $\alpha_2$ is on $\ell$.  Also, $\alpha_2$ has a
neighbour $\alpha_3$ that is at distance $d-1$ from $\alpha$, so by induction, $\alpha_3$ is on $\ell$.  Now if
$\alpha_1$ were not on $\ell$, $\alpha_2$ would have two neighbours $\alpha_1$ and $\alpha_3$, for which $|\alpha_2 \cap
\alpha_1|=|\alpha_2\cap \alpha_3|=i$ but $\alpha_1$ and $\alpha_3$ are not on the same line, so by vertex-transitivity, the same
would have to be true for $\alpha$, contradicting our assumption.  This completes the induction.
\qed\end{proof}

\begin{lemma}\label{A_n-bound-n}
Under Hypothesis~\ref{hyp-intrans}, $n \le 3k-2k_1$, where $k_1$ is the maximum
cardinality of the intersection of two collinear points.
\end{lemma}

\begin{proof}
Towards a contradiction, suppose that $n \ge 3k-2k_1+1$.

By Lemma~\ref{share-i-collinear} every two points whose intersection has cardinality $k_1$ will be
collinear.  Let $\alpha=\{1, \ldots, k\}$ and $\beta=\{1,\ldots, k_1, k+1, \ldots,
2k-k_1\}$ be two collinear points.

If $k_1=k-1$ then $\beta=\{1,\ldots, k-1, k+1\}$ and by Lemma~\ref{k-1-non-collinear-pt} there is some
$\beta'$ such that $|\alpha \cap \beta'|=k-1$ but $\beta'$ is not collinear with $\beta$.  Now, since $\beta'$ is not collinear with $\beta$, Lemma~\ref{share-i-collinear} tells us that $|\beta \cap \beta'| \neq k-1$, so
without loss of generality, $\beta'=\{2,\ldots, k, k+2\}$.  Now consider the point $\alpha'=\{1, \ldots,
k-1,k+2\}$.  The cardinality of the intersection of this set with $\alpha$, $\beta$, and $\beta'$ is
$k-1=k_1$ (in each case), so by Lemma~\ref{share-i-collinear}, $\alpha'$ is collinear with both $\alpha$ and
$\beta$ so must be on the unique line joining $\alpha$ and $\beta$. Now $\alpha'$ is also collinear with both $\alpha$ and
$\beta'$ so must be on the unique line joining $\alpha$ and $\beta'$; since these two lines are not the same and $\alpha$ is their point of intersection, we
have a contradiction.

Thus we may assume $k_1 \le k-2$.  Now we let $\beta'=\{1, \ldots, k_1, k+2, \ldots, 2k-k_1+1\}$. The
cardinality of $\alpha \cap \beta'$ is $k_1$, so $\alpha$ and $\beta'$ are collinear by Lemma~\ref{share-i-collinear};
the cardinality of $\beta \cap \beta'$ is $k-1>k_1$, so by definition of $k_1$, $\beta$ and $\beta'$ cannot be
collinear.  Consider the point $\alpha'=\{1, \ldots, k_1, 2k-k_1+2, \ldots, 3k-2k_1+1\}$.  (Since $n \ge
3k-2k_1+1$ this set is a point of $\Q$.)  The cardinality of the intersection of this set with
$\alpha$, $\beta$, and $\beta'$ is $k_1$ (in each case), so by Lemma~\ref{share-i-collinear}, $\alpha'$ is collinear
with $\alpha$, $\beta$, and $\beta'$, not all of which are collinear; as before, this is a contradiction.
\qed\end{proof}

\begin{theorem}\label{A_n-intrans-done}
Under Hypothesis~\ref{hyp-intrans}, $n=6$ and $s=t=2$.
\end{theorem}

\begin{proof}
Taking $k$ to be the smaller of $k$ and $n-k$, we can assume without loss of generality that $n \ge
2k$.  As before, let $k_1$ be the maximum cardinality of the intersection of two collinear
points. By Lemma~\ref{A_n-bound-n}, we have $n \le 3k-2k_1$.  Combining this with $n \ge 2k$ yields $k_1 \le k/2$.

 By Lemma~\ref{share-i-collinear} every two points whose intersection has cardinality $k_1$ will be
 collinear.  Let $\alpha=\{1, \ldots, k\}$ and $\beta=\{1,\ldots, k_1, k+1, \ldots,
 2k-k_1\}$ be two collinear points. Now, since $\Q$ is thick, there are at least three points on each line.  Each point is a set of $k$
elements, and their pairwise intersections are at most $k_1$, so by inclusion-exclusion, the total
number of elements is at least $3k-3k_1$.  That is, we have $n \ge 3k-3k_1$. Thus, as in the second part of the proof of  Lemma~\ref{A_n-bound-n},  $\beta'=\{1, \ldots, k_1, k+2, \ldots, 2k-k_1+1\}$ is a point of $\mathcal{Q}$. Observe that $\alpha$  is collinear with both $\beta$ and $\beta'$, but $\beta$ and $\beta'$ are not collinear.

Consider the point $\alpha'=\{k_1+1, \ldots, 2k_1, k+1, \ldots, k+k_1, 2k-k_1+1, \ldots, 3k-3k_1\}$.
Since $k_1 \le k/2$, this is a set of $k_1+k_1+(k-2k_1)=k$ distinct elements from $\{1, \ldots,
n\}$, so is a point of $\Q$.  If $k_1\ge 1$ then the cardinality of the intersection of this set
with $\alpha$, $\beta$, and $\beta'$ is $k_1$ (in each case), so by Lemma~\ref{share-i-collinear}, $\alpha'$ is
collinear with $\alpha$, $\beta$, and $\beta'$, not all of which are collinear; as in the proof of Lemma~\ref{A_n-bound-n}, this is a
contradiction.  On the other hand, if $k_1=0$ then $3k-3k_1 \le n \le 3k-2k_1$ implies $n=3k$.  In
this case, there can be only 3 points on any line since there are only 3 pairwise disjoint sets of
cardinality $k$ in a set of cardinality $3k$, so $s=2$, forcing (by Higman's inequality, Lemma~\ref{lem:s+tdiv}(i)) $t \le 4$.  Hence the number of
points in $\Q$ is at most $27$. We also know that the number of points is $\binom{n}{k}=\binom{n}{n/3}$.  Since $\binom{9}{3}>27$
and $\Q$ is thick, we obtain $n = 6$, which gives $t=2$ as claimed.  \qed\end{proof}

%%%%%%%%%%%%%%%%%%%%%%%%%
%
%   Imprimitive case
%
%%%%%%%%%%%%%%%%%%%%%%%%%

\subsection{Imprimitive stabiliser}
In this subsection we deal with the possibility that $G$ is $A_n$ or $S_n$, acts primitively on the
set of points of the generalised quadrangle $\Q$, and the stabiliser of a point of the generalised
quadrangle is imprimitive in its action on $\{1, \ldots, n\}$. In this case we can identify the
points of $\Q$ with partitions of a set of cardinality $n$ into $b$ parts of cardinality $a$, where
$ab=n$ and $a, b >1$.

We first show that we can assume $b \ge 3$, since $b=2$ reduces to the intransitive case dealt with
in the previous subsection.

\begin{lemma}\label{imprim-b-neq-2}
Suppose that $G=A_n$ or $S_n$ acts primitively on the set of points of a thick generalised quadrangle and
$G_\alpha$ is the stabiliser of a partition of $\{1,\ldots,n\}$ into $b$ parts. Then $b\neq 2$.
\end{lemma}

\begin{proof}
Towards a contradiction, suppose that $b=2$.  Then $n$ is even, and $a=n/2$.  Since $G_\alpha$ is imprimitive in its action on $\{1, \ldots, n\}$, it acts transitively on $\{1, \ldots, n\}$, so $G=G_1 G_\alpha$.  Hence $G_1$ acts
transitively on the points of $\Q$.

Now, $G_1 \cap G_\alpha= (S_{n/2-1} \times S_{n/2})\cap G$, which is maximal in $G_1$, so $G_1$ acts
primitively on the points of $\Q$, and $\soc(G_1)$ is $A_{n-1}$.  Furthermore, in this action, the
point stabiliser $G_1 \cap G_\alpha$ is intransitive on $\{2, \ldots, n\}$.  Theorem~\ref{A_n-intrans-done} then implies that $n-1=6$, contradicting $n$ being even.  \qed\end{proof}

We summarise in the following hypothesis the situation that will be assumed in almost all of the results
within this subsection.

\begin{hyp}\label{imprim-hyp}
Suppose that $G=A_n$ or $S_n$, $G$ acts primitively on the points of a thick generalised quadrangle $\Q$ of order $(s,t)$, $\alpha$ is a point of $\Q$, and the action of $G_\alpha$ on $\{1, \ldots, n\}$ is imprimitive, so the points
of $\Q$ are identified with partitions of $\{1, \ldots, n\}$ into $b$ parts of cardinality $a$ with
$b\ge 3$.
\end{hyp}

Our approach will be to use Lemma~\ref{sub-GQ-theta} to produce a substructure of $\Q$, or in some
cases, two nested substructures.  First we must show that the substructures produced are in fact thick
generalised quadrangles rather than any of the degenerate cases.  Then we will use
Lemma~\ref{GQ-book-2} and Lemma~\ref{sub-GQ-order} to produce bounds on $n$. This will reduce the
problem to a finite one, and we will use various means to eliminate the remaining possibilities.

\begin{notn}\label{thetas}
Through the rest of this subsection, we will be using two permutations, $\theta_1$ and $\theta_2$.
If $a \ge 3$ then $\theta_1= (1\ 2\ 3)$, and if $a \ge 4$ then $\theta_2=(1\ 2\ 4)$ while if $a=3$,
$\theta_2=(4\ 5\ 6)$.  If $a=2$ then $\theta_1=(1\ 2)(3\ 4)$ and $\theta_2=(1\ 2)(5\ 6)$.  So we
always have $\theta_1, \theta_2 \in A_n \le G$.  Starting with the generalised quadrangle $\Q$, we let $\Q'$ denote the substructure of elements of $\Q$ that are fixed by $\theta_1$, and $\Q''$ the substructure of elements of $\Q$ that are fixed by both $\theta_1$ and $\theta_2$.  We use $v$, $v'$, and $v''$ to denote the number of points of $\Q$, $\Q'$, and $\Q''$ (respectively).
\end{notn}

With this notation in place, we are ready to begin showing that the degenerate cases do not arise in
the situations that interest us.  First we examine some parameters of the substructures and the groups that act on them.

\begin{lemma}\label{v-values}\label{sub-GQ-gp-actions}
Under Hypothesis \ref{imprim-hyp} and using Notation
\ref{thetas}, the values for $v''$, $v/v'$ and $v'/v''$ are given by Table
\ref{tab:ratios}. Table~\ref{tab:actions} records information about the actions of certain subgroups 
of $G$ acting on $\Q'$ and $\Q''$ when $n \ge 15$ that will be needed later.

\newcommand\T{\rule{0pt}{2.6ex}}
\newcommand\B{\rule[-1.2ex]{0pt}{0pt}}
\begin{table}[H]
\begin{center}
\begin{tabular}{cccc}
  $a$ & $v''$ & $v/v'$ & $v'/v''$\\
  \hline 
 $\ge 4$\T &  $\frac{(n-4)!}{(a-4)!(a!)^{b-1}(b-1)!}$ & $\frac{(n-1)(n-2)}{(a-1)(a-2)}$ & $\frac{n-3}{a-3}$ \\ \\
 $3$ & $\frac{(n-6)!}{6^{n/3-2}(n/3-2)!}$ & $\frac{(n-1)(n-2)}{2}$ &$\frac{(n-4)(n-5)}{2}$ \\ \\
 $2$&$\frac{(n-6)!}{2^{n/2-3}(n/2-3)!}$ & $\frac{(n-1)(n-3)}{3}$ & $3(n-5)$\\
 \hline
\end{tabular}
\caption{Values of $v$, $v'$ and $v''$}
\label{tab:ratios}
\end{center}
\end{table}
\begin{table}[H]
\begin{center}
\begin{tabular}{ccc}
  $a$ & $\Q'$ & $\Q''$\\
  \hline
 $\ge 4$ &  $A_{n-3}$ (transitive) & $A_{n-4}$ (transitive) \\
 $3$ & $A_{n-3}$ (primitive) & $A_{n-6}$ (primitive) \\
 $2$ &$(S_4 \times S_{n-4})\cap A_n$ (transitive) & $A_{n-6}$ (primitive) \\
 $2$ &$A_{n-4}$ (3 orbits of equal length)
\end{tabular}
\caption{Certain subgroups of $G$ and their actions on $\Q'$ and $\Q''$ when $n \ge 15$.}
\label{tab:actions}
\end{center}
\end{table}
\end{lemma}

\begin{proof}
First note that $v=\frac{n!}{(a!)^bb!}$. For $a\ge 3$, the points of $\mathcal{Q}'$ are the
partitions of $\{1,\ldots,n\}$ that have $\{1,2,3\}$ together in one part. The group
$\Sym(\{4,5,\ldots,n\})$ acts transitively on the set of  these partitions with the stabiliser of one such
partition isomorphic to $S_{a-3}\times (S_a\Wr S_{b-1})$. Thus
$v'=\frac{(n-3)!}{(a-3)!(a!)^{b-1}(b-1)!}$.  Furthermore, $\Sym(\{4,5,\ldots,n\})\cap\soc(G)$ is isomorphic to $A_{n-3}$ and still acts transitively on the set of partitions, as listed in the first row of Table~\ref{tab:actions}.  Furthermore, if $a=3$, this action is primitive (see for example \cite{LPS}).
For $a\ge 4$, the points of $\mathcal{Q''}$ are the
partitions of $\{1, \ldots, n\}$ that have $\{1,2,3,4\}$ together in one part, and so $v'$ is as given in the
first row of Table \ref{tab:ratios}.  The values for $v/v'$ and $v'/v''$ follow.  The subgroup of $\soc(G)$ that fixes 1, 2, 3, and 4 is isomorphic to $A_{n-4}$ and acts transitively on the points of $\Q''$.

If $a=3$ then the points of $\mathcal{Q''}$ are the partitions of $\{1,\ldots, n\}$ that have
$\{1,2,3\}$ as one part and $\{4, 5, 6\}$ as another. The group $\Sym(\{7,8,\ldots,n\})$ acts
transitively on the set of such partitions with the stabiliser of one such partition isomorphic to
$S_3\Wr S_{b-2}$. Thus $v''$ is as given in the second row of Table \ref{tab:ratios} and the values
of $v/v'$ and $v'/v''$ follow.  Also, $\Sym(\{7,8,\ldots,n\})\cap \soc(G)$ is isomorphic to $A_{n-6}$ and (similar to the action on $\Q'$ discussed above) acts primitively on these partitions.

Finally, when $a=2$, the points of $\mathcal{Q}'$ are the partitions of $\{1,\ldots,n\}$ that have
two parts contained in $\{1,2,3,4\}$. The group $\Sym(\{1,2,3,4\})\times \Sym(\{5,6,\ldots,n\})$
acts transitively on the set of such partitions with the stabiliser of one such partition isomorphic
to $D_8\times (S_2\Wr S_{b-2})$. Thus $v'=\frac{3(n-4)!}{2^{b-2}b!}$. 
It is straightforward to see that $A_{n-4}$ (acting on $\{5, \ldots,
n\}$) has 3 orbits of equal length on the points of $\Q'$, consisting of all partitions that have 1
and 2 in the same part; those that have 1 and 3 in the same part; and those that have 1 and 4 in the
same part.  Since $(\Sym(\{1,2,3,4\})\times \Sym(\{5,6,\ldots,n\}))\cap\soc(G)$
permutes these orbits, it acts transitively on the points of $\Q'$.

Also when $a=2$, the points of $\mathcal{Q}''$
are the partitions of $\{1, \ldots, n\}$ that include $\{1,2\}$, $\{3,4\}$ and $\{5,6\}$. The group
$\Sym(\{7,8,\ldots,n\})$ acts transitively on such partitions with stabiliser isomorphic to $S_2\Wr
S_{b-3}$ and so $v''$ is as given in the third row of Table \ref{tab:ratios}. The values of $v/v'$
and $v/v''$ follow.
The action of $\Sym(\{7,8,\ldots,n\})\cap\soc(G)$ on the points of $Q''$ is analogous to the action of $A_{n-6}$ on
the points of $\Q''$ when $a=3$. This action is primitive since
$b-3\neq 4$ \cite{LPS}, as $n \ge 15$.
\qed\end{proof}

%\begin{lemma}\label{points-exist}
%Under Hypothesis~\ref{imprim-hyp}, let $\theta_1$ and $\theta_2$ be as given in Notation~\ref{thetas}.  Then if $n \ge 10$, both the incidence structure of elements of $\Q$ fixed by $\theta_1$, and the substructure of elements fixed by both $\theta_1$ and $\theta_2$, have at least 2 points.
%\end{lemma}

%\begin{proof}
%Let $N''$ be the number of points of the substructure of elements fixed by both $\theta_1$ and $\theta_2$.  We aim to show that $N''\ge 2$.  Since every point of this substructure is fixed by $\theta_1$, this will also show that there are at least 2 points in the incidence structure of elements fixed by $\theta_1$.

%It is straightforward to count the number of partitions that are fixed by both $\theta_1$ and $\theta_2$.  If $a \ge 4$, then this will be the number of partitions of $\{1, \ldots, n\}$ that have 1, 2, 3 and 4 in the same part, which is $\frac{(n-4)!}{(a-4)!(a!)^{b-1}(b-1)!}$.  When $n>8$, this is at least 2.  If $a=3$, this will be the number of partitions of $\{1,\ldots, n\}$ that have 1, 2 and 3 in one part and 4, 5 and 6 in another part.  This is $\frac{(n-6)!}{(3!)^{n/3-2}(n/3-2)!}$, which is at least 2 when $n >9$. If $a=2$, this will be the number of partitions of $\{1, \ldots, n\}$ that have 1 and 2 in one part, 3 and 4 in a second part, and 5 and 6 in a third part.  This is $\frac{(n-6)!}{(2!)^{n/2-3}(n/2-3)!}$, which is at least 2 when $n>8$.
%\qed\end{proof}

For some of the results to come in this subsection, the following bound will prove useful.

\begin{lemma}\label{bound-N}
When $a \ge 4$ and $b \ge 3$ and $n \ge 16$ we have
$\frac{n!}{(a!)^bb!} \ge (2.2)^n$.
\end{lemma}

\begin{proof}
Stirling's formula \cite{robbins} gives us that for every $n \ge 1$, 
$$\sqrt{2\pi n}\ e^{1/(12n+1)}(n/e)^n \le n! \le \sqrt{2\pi n}\ e^{1/(12n)}(n/e)^n.$$

First suppose that $b \ge 5$.  Let $r=4/(24^{1/4})<2$.  It is straightforward to check that $a! \le
(a/r)^a$ since $a \ge 4$ and the two are equal when $a=4$. Hence we have (using Stirling's formula
to get $n! \ge (n/e)^n$ and the upper bound we have just calculated for $a!$)
$$\frac{n!}{(a!)^b b!} \ge \frac{(n/e)^n}{(a/r)^n b!}=\frac{(br/e)^n}{b!} \ge (r/e)^n b^{n-b} \ge
\left(\frac{r b^{3/4}}{e}\right)^n$$ since $a \ge 4$ implies $b \le n/4$, so $n-b \le 3n/4$.  Since
$b \ge 5$, it is easy to check that $b^{3/4}r/e \ge 2.2$, yielding the desired result.

If $b=4$ then $a=n/4$ and we use Stirling's formula on $n$ and $a$ to obtain
\begin{eqnarray*}
\frac{n!}{(a!)^b b!} & \ge & \frac{\sqrt{2\pi n}\ e^{1/(12n+1)}(n/e)^n}{\sqrt{(2\pi n)/4}^4 e^{4/(3n)}(n/4e)^n 4!}\\
&=&\frac{2 e^{1/(12n+1)}}{3\sqrt{2\pi n}^3 e^{4/(3n)}}4^n
\end{eqnarray*}
and it can be verified computationally that this is at least $(2.2)^n$ since $a \ge 4$ implies $n
\ge 16$.

If $b=3$ then $a=n/3$ and we use Stirling's formula on $n$ and $a$ to obtain
\begin{eqnarray*}
\frac{n!}{(a!)^b b!} & \ge & \frac{\sqrt{2\pi n}\ e^{1/(12n+1)}(n/e)^n}{\sqrt{2\pi n/3}^3 e^{3/(4n)}(n/3e)^n 3!}\\
&=&\frac{\sqrt{3} e^{1/(12n+1)}}{4\pi n e^{3/(4n)}}3^n
\end{eqnarray*}
and it can be verified computationally that this is at least $(2.2)^n$ when $n \ge 16$.  \qed\end{proof}

We now show that there must always be points and lines fixed by $\theta_1$ and by $\theta_2$.

\begin{lemma}\label{lines-exist}\label{points-exist}
Under Hypothesis~\ref{imprim-hyp} and using Notation~\ref{thetas}, if $n\ge 10$ then both $\mathcal{Q}'$ and $\mathcal{Q}''$ contain at
least two points. Moreover, if $n \ge 37$ then both $\Q'$ and $\Q''$ have at least
one line. 
\end{lemma}

\begin{proof}
The three values for $v''$ given in Table \ref{tab:ratios} are at least 2 for $n\ge 10$. Also every
point of $\mathcal{Q}''$ is contained in $\mathcal{Q}'$. Thus for $n\ge 10$ both $\mathcal{Q}'$ and
$\mathcal{Q}''$ have at least two points.  

Any line fixed by both $\theta_1$ and $\theta_2$ will be fixed by $\theta_1$, so if we show that
$\Q''$ has at least one line, we will be done.

If there were two collinear points of $\Q$
fixed by $\theta_1$ and by $\theta_2$, then the line containing them would also be fixed by
$\theta_1$ and by $\theta_2$.  So if there are no lines in $\Q''$, we must have
$v''$ pairwise non-collinear points in $\Q$.  The number of pairwise non-collinear points in a
generalised quadrangle of order $(s,t)$ is at most $st+1$, so we have $st+1 \ge v''$. Thus we have
$v/(s+1)=st+1 \ge v''$.

By Lemma~\ref{v-values}. the values for $r=v/v''$ can be calculated quickly using Table~\ref{tab:ratios}. Now $v/(s+1)\ge v''=v/r$ implies $r\ge s+1>v^{1/4}$ by
Lemma~\ref{useful}(ii).  Calculations using the lower bound for $v$ given by Lemma~\ref{bound-N} if
$a \ge 4$ (we have $b\ge 3$ by Lemma~\ref{imprim-b-neq-2}) and the formula for $v$ if $a$ is 2 or 3,
show that this inequality forces $n$ to be less than 50: if $a\ge 4$ then $n \le 49$; if $a=3$ then
$n \le 36$; and if $a=2$ then $n \le 32$.  When $a \ge 4$, we can evaluate the inequality again
using each divisor $a$ of $n$ for $n \le 49$ and the corresponding formula for $N$, to see that we
must actually have $n \le 24$.  \qed\end{proof}

\begin{lemma}\label{points-off-line}
Under Hypothesis~\ref{imprim-hyp} and using
Notation~\ref{thetas}, if $n \ge 16$ then neither $\Q'$ nor $\Q''$ has
all of its points on one line.
\end{lemma}

\begin{proof}
If all points of $\Q''$ are on a single line, then there are at most $s+1$ such points, so we have $s+1
\ge v''$. Since $v'\ge v''$, if this inequality leads to a contradiction, it will clearly not be possible for all points of $\Q'$ to be on a single line either. Letting $r=v/v''$ it follows that $r(s+1)\geq v$ and so by Lemma~\ref{useful}(iii),
$r(v^{2/5}+1)>v$. Thus $r>v/(v^{2/5}+1) >v^{1/5}(v^{2/5}-1)$.

By Lemma~\ref{v-values}, the values for $r$ can easily be calculated using Table~\ref{tab:ratios}.  Calculations using
the lower bound for $v$ given by Lemma~\ref{bound-N} if $a \ge 4$ (note that $b\ge 3$) and the
formula for $v$ if $a$ is 2 or 3, show that the inequality $r>v^{1/5}(v^{2/5}-1)$ forces $n$ to be
less than 16: if $a\ge 4$ then $n \le 13$ (so since $n=ab$ we must have $n=12$); if $a=3$ then $n
\le 15$; and if $a=2$ then $n \le 14$.  \qed\end{proof}

\begin{cor}\label{lines-off-point}
Under Hypothesis~\ref{imprim-hyp} and using
Notation~\ref{thetas}, if $n \ge 16$, then neither $\Q'$ nor $\Q''$ has
all of its lines pass through one point.
\end{cor}

\begin{proof}
By Lemma~\ref{points-exist}, both $\Q'$ and $\Q''$ have at least 2 points.  By
Lemma~\ref{sub-GQ-gp-actions}, both are point-transitive.  Therefore if all of the lines were to
pass through one point, all of the lines would have to pass through every point.  Since there is a
unique line through any two points, this is only possible if the incidence structure has at most one
line.  By Lemma~\ref{points-off-line}, this cannot occur.  \qed\end{proof}

\begin{lemma}\label{not-dual-grid}
Under Hypothesis~\ref{imprim-hyp} and using
Notation~\ref{thetas}, if $n \ge 9$, then neither $\Q'$ nor $\Q''$ is a
dual grid.
\end{lemma}

\begin{proof}
Towards a contradiction, suppose that either $\Q'$ or $\Q''$ is a dual grid.  By
Lemma~\ref{sub-GQ-gp-actions} and Table~\ref{tab:actions}, except for $\Q'$ in the case $a=2$, there is some $m \ge n-6$ such
that $A_m$ acts transitively on the points of the substructure, and hence acts transitively on the
lines of the dual structure, which is a grid.  The lines in a grid can be (uniquely) partitioned
into two sets (think of them as the horizontal lines and the vertical lines) such that the lines
within each set are pairwise non-concurrent, but if one line is chosen from each set, they must
intersect. Since $A_m$ acts transitively on the lines, the subgroup of $A_m$ that fixes each of
these sets (setwise) is a subgroup of index 2 in $A_m$. Since $m \ge n-6 \ge 3$, $A_m$ does not have an index 2 subgroup, a contradiction.

In the case of $\Q'$ when $a=2$, again by Lemma~\ref{sub-GQ-gp-actions}, this structure is
point-transitive, so the dual structure, a grid, is line-transitive.  Hence the two parts in the
unique partition of the lines described above (horizontal and vertical) have the same cardinality.
The action of $A_{n-4}$ on this structure has 3 orbits of equal length on the lines
(Lemma~\ref{sub-GQ-gp-actions} and Table~\ref{tab:actions}).  Since the cardinalities of the two sets of lines are equal, there
must be an element of $A_{n-4}$ that interchanges the two sets.  Again, this forces the subgroup
of $A_{n-4}$ that fixes each of these sets of lines (setwise) to be an index 2 subgroup of
$A_{n-4}$, a contradiction.  \qed\end{proof}

\begin{lemma}\label{not-grid}
Under Hypothesis~\ref{imprim-hyp} and using
Notation~\ref{thetas},  if $n \ge 15$ then neither $\Q'$ nor $\Q''$ is a grid.
\end{lemma}

\begin{proof}
Towards a contradiction, suppose that either $\Q'$ or $\Q''$ is a grid.  As in the proof of
Lemma~\ref{not-dual-grid}, the lines can be (uniquely) partitioned into two sets (think of them as
the horizontal lines and the vertical lines) such that the lines within each set are pairwise
non-concurrent, but if one line is chosen from each set, they must intersect.  By
Lemma~\ref{sub-GQ-gp-actions} and Table~\ref{tab:actions}, except for $\Q'$ in the case $a=2$, there is some $m \ge n-6$ such
that $A_m$ acts transitively on the points of the substructure.  If there were an element of $A_m$
that interchanges the two sets of lines, then the subgroup of $A_m$ that fixes each set of lines
(setwise) would have index 2, a contradiction.  So every element of $A_m$ fixes each of the sets of
lines.  Since $A_m$ acts transitively on the points, it must act transitively on each of the two
sets of lines.  Thus, each set of lines forms a system of imprimitivity on the points of the
structure.  This is clearly impossible when the action of $A_m$ on the points is primitive; by
Lemma~\ref{sub-GQ-gp-actions} and Table~\ref{tab:actions}, this deals with the case $a=3$, and with $\Q''$ when $a=2$ or $a=4$.

If $a \ge 4$ and $x$ is a point of the substructure, then the stabiliser of $x$ in $A_m$ is
$(S_{a-(n-m)} \times (S_a \Wr S_{(n-a)/a}))\cap A_m$.  Now, by \cite{LPS}, $S_a \Wr S_{(n-a)/a}$ is a maximal
subgroup of $S_{n-a}$, so there is a unique group that lies properly between $(S_{a-(n-m)} \times
(S_a \Wr S_{(n-a)/a}))\cap A_m$ and $A_m$, namely $(S_{a-(n-m)} \times S_{n-a}) \cap A_m$. When
$a=4$ and we are dealing with $\Q''$, we have $a-(n-m)=0$ so this group is actually $A_m$ and the
action is primitive, as previously claimed.  In every other case, this corresponds to a unique
system of imprimitivity, contradicting the existence of two systems that we proved in the preceding
paragraph.

The only remaining substructure that could be a grid is $\Q'$ when $a=2$.  By
Lemma~\ref{sub-GQ-gp-actions} and Table~\ref{tab:actions}, $(S_4 \times S_{n-4})\cap A_n$ acts transitively on the points, and
the point stabiliser is $(D_8\times (S_2 \Wr S_{(n-4)/2}))\cap A_n$.  There are two groups that lie
properly between $(D_8\times (S_2 \Wr S_{(n-4)/2}))\cap A_n$ and $(S_4 \times S_{n-4})\cap A_n$:
namely, $(D_8 \times S_{n-4})\cap A_n$ and $(S_4\times(S_2 \Wr S_{(n-4)/2}))\cap A_n$.  These must
correspond to the systems of imprimitivity formed by the two sets of lines.  Now, $(D_8 \times
S_{n-4})\cap A_n$ gives blocks of cardinality $\frac{(n-4)!}{2^{(n-4)/2}(\frac{n-4}{2})!}$, while
$(S_4\times(S_2 \Wr S_{(n-4)/2}))\cap A_n$ gives blocks of cardinality 3.  Thus, there are lines in
$\Q'$ that contain $\frac{(n-4)!}{2^{(n-4)/2}(\frac{n-4}{2})!}$ points, and since $\Q'$ is a
substructure of $\Q$, we must have $s+1 \ge \frac{(n-4)!}{2^{(n-4)/2}(\frac{n-4}{2})!}$.  Using
Lemma~\ref{useful}(iii), we get
$$\left(\frac{n!}{2^{n/2}(n/2)!}\right)^{2/5}+1 > s+1 \ge  \frac{(n-4)!}{2^{(n-4)/2}(\frac{n-4}{2})!}.$$
This is satisfied only when $n \le 10$.
\qed\end{proof}

We now have bounds on $n$ whenever the substructures are degenerate, so are ready to produce bounds
for $n$ in the non-degenerate cases.

\begin{lemma}\label{imprim-bound-n}
Under Hypothesis~\ref{imprim-hyp}, we have $n \le 36$.
\end{lemma}

\begin{proof}
Towards a contradiction, suppose that $n \ge 37$.  We use
Notation~\ref{thetas}.

%We will need the ratios $N/N'$ and $N'/N''$ where $N$ is the number of points of $\Q$, $N'$ is the
%number of points of $\Q'$, and $N''$ is the number of points of $\Q''$.  Calculations like those
%performed in the proof of Lemma~\ref{lines-exist} show that: \begin{itemize} \item if $a \ge 4$,
%$N/N'=(n-1)(n-2)/(a-1)(a-2)$ and $N'/N''=(n-3)/(a-1)$; \item if $a=3$, $N/N'=(n-1)(n-2)/2$ and
%$N'/N''=(n-4)(n-5)/2$; and \item if $a=2$, $N/N'=(n-1)(n-3)/3$ and $N'/N''=3(n-5)$.  \end{itemize}

Lemma~\ref{sub-GQ-theta} tells us that each of these incidence structures falls into one of 7
categories.  The first and second are ruled out by Lemma~\ref{lines-exist}.  The third is eliminated by Lemma~\ref{lines-off-point}, and the fourth by
Lemma~\ref{points-off-line}.  Lemmas~\ref{not-grid} and~\ref{not-dual-grid} tell us that the fifth
and sixth (respectively) cannot occur.  Thus, both $\Q$ and $\Q'$ are subquadrangles of orders
$(s',t')$ and $(s'', t'')$ respectively, and $s', s'', t', t'' \ge 2$.  Now, Lemma~\ref{GQ-book-2}
tells us that if $s''=s'=s$ then $t''=1$, contradicting $t'' \ge 2$, so this cannot occur.

Suppose that $s'=s$.  Then we have just seen that we must have $s''<s$, so by
Lemma~\ref{sub-GQ-order}, $s''t''\le s'=s$.  Hence $r=v'/v'' >t'\ge s^{1/2} >(v^{1/4}-1)^{1/2}$ (the
last two inequalities come from Higman's inequality \ref{lem:s+tdiv}(i) and Lemma~\ref{useful}(ii), while the first is
Lemma~\ref{useful}(iv) applied to $v'$ and $v''$ rather than to $v$ and $v'$).  
By Lemma~\ref{v-values}, the values for $r$ are given in Table~\ref{tab:ratios}.
Calculations using the lower bound for $v$ given by Lemma~\ref{bound-N} if
$a \ge 4$ (we have $b\ge 3$ by Lemma~\ref{imprim-b-neq-2}) and the formula for $v$ if $a$ is 2 or 3,
show that this inequality forces $n$ to be less than 36: if $a\ge 4$ then $n \le 35$; if $a=3$ then
$n \le 33$; and if $a=2$ then $n \le 28$.  This is a contradiction. 

So we must have $s'<s$.  Thus by Lemma~\ref{sub-GQ-order}, $s't' \le s$.  Hence $r'=v/v'>t>v^{1/5}-1$
(from parts (iv) and (i) of Lemma~\ref{useful}). 
 By Lemma~\ref{v-values}, the values for $r'$ are given in Table~\ref{tab:ratios}.
Calculations using the lower bound for $v$ given by Lemma~\ref{bound-N} if
$a \ge 4$ (we have $b\ge 3$ by Lemma~\ref{imprim-b-neq-2}) and the formula for $v$ if $a$ is 2 or 3,
show that this inequality forces $n$ to be less than 33: if $a\ge 4$ then $n \le 32$; if $a=3$ then $n \le 21$; and if $a=2$ then
$n \le 22$. This is a contradiction.  \qed\end{proof}

\begin{cor}\label{cor:imprim}
Under Hypothesis~\ref{imprim-hyp}, $n=6$ and
$\Q$ is the unique generalised quadrangle of order $(2,2)$.
\end{cor}

\begin{proof}
By Lemma~\ref{imprim-bound-n}, we have $n \le 36$. The following table shows the only instances 
where the number $v$ of points of $\Q$ satisfies
\begin{align*}
v&=\frac{n!}{(a!)^b b!},\quad\text{for some }a,b>1\\
v&=(s+1)(st+1),\quad\text {for some }s,t>1, \sqrt{t}\le s\le t^2, s+t \text{ divides } st(s+1)(t+1):
\end{align*}
\begin{center}
\begin{tabular}{ccccc}
$n$&$s$&$t$&$a$&$b$\\
\hline
6&2&2&2&3\\
9&9&3&3&3\\ 
10&8&13&2&5\\
16&76&449&4&4 
\end{tabular}
\end{center}
The first case is the unique generalised quadrangle of order $(2,2)$ (c.f.,
\cite[5.2.3]{Payne:2009ly}).  The last case cannot occur, since in this event, we would have (by
Lemmas~\ref{sub-GQ-theta}, \ref{points-exist}, \ref{points-off-line}, \ref{lines-off-point}, \ref{not-dual-grid}
and~\ref{not-grid}) that the substructure $\Q'$ either has no lines, or is non-degenerate of order
$(s',t')$ with $s', t' \ge 2$.  Furthermore, $\Q'$ has $v'=75075$ points.  Now,
$st+1=76(449)+1=34125<v'$ so by the same reasoning used in Lemma~\ref{lines-exist}, $\Q'$ does have
lines.  However, there are no integers $s',t'\ge 2$ such that $(s'+1)(s't'+1)=75075$.

For $n= 9$, a simple computation shows that $G$ does have a subdegree equal to $s(t+1)=36$, but while the associated graph
(which would be the putative collinearity graph of the generalised quadrangle) has diameter 2, it is not strongly-regular.
For $n=10$, there is no union of suborbits yielding a set of cardinality $s(t+1)=112$; the nontrivial subdegrees are
$\{20, 60, 80, 160, 240, 384\}$ (in both the symmetric and alternating cases).

Therefore, by Lemma \ref{imprim-bound-n}, $n=6$ is the only possibility that occurs.
\qed\end{proof}

%%%%%%%%%%%%%%%%%%%%%%%%%
%
%   n <= 47  case
%
%%%%%%%%%%%%%%%%%%%%%%%%%

\subsection{Small degree actions of $A_n/S_n$, for $n\le 47$}

\begin{table}[H]
\begin{center}
\begin{tabular}{|c|c|c|c|}
\hline
 $D_{10}\le A_{5}$  & $\mathsf{AGL}(1, 5)\le S_{5}$  & $\mathsf{PSL}(2,5)\le A_{6}$  & $\mathsf{PGL}(2,5)\le S_{6}$ \\ 
$\mathsf{PSL}(3, 2)\le A_{7}$  &  $\mathsf{AGL}(1, 7)\le S_{7}$  & $\mathsf{ASL}(3, 2)\le A_{8}$  & $\mathsf{PGL}(2, 7)\le S_{8}$  \\
$\mathsf{ASL}(2,3)\le A_{9}$  & $\mathsf{P\Gamma L}(2, 8)\le A_{9}$ & $\mathsf{AGL}(2, 3)\le S_{9}$ & $M_{10}\le A_{10}$ \\
 $\mathsf{P\Gamma L}(2, 9) \le S_{10}$  & $M_{11} \le A_{11}$  & $\mathsf{AGL}(1, 11) \le S_{11}$  & $M_{12} \le A_{12}$ \\ 
 $\mathsf{PGL}(2, 11) \le S_{12}$  & $13:6 \le A_{13}$  & $\mathsf{PSL}(3, 3) \le A_{13}$  & $\mathsf{AGL}(1, 13) \le S_{13}$ \\ 
$\mathsf{PSL}(2,13) \le A_{14}$  & $\mathsf{PGL}(2,13) \le S_{14}$  & $\mathsf{PSL}(4, 2) \le A_{15}$  & $2^4.\mathsf{PSL}(4, 2) \le A_{16}$ \\ 
$\mathsf{P\Gamma L}(2, 2^4) \le A_{17}$  & $\mathsf{AGL}(1, 17) \le S_{17}$  & $\mathsf{PSL}(2,17) \le A_{18}$  & $\mathsf{PGL}(2,17) \le S_{18}$ \\ 
$\mathsf{PSL}(2,19) \le A_{20}$  & $\mathsf{PGL}(2,19) \le S_{20}$  & $A_7 \le A_{21}$  & $\mathsf{PGL}(3, 4) \le A_{21}$ \\ 
$S_7\le S_{21}$  & $\mathsf{P\Gamma L}(3, 4) \le S_{21}$  & $M_{22} \le A_{22}$  & $M_{22}:2 \le S_{22}$ \\ 
$M_{23} \le A_{23}$  & $M_{24} \le A_{24}$  & $\mathsf{PGL}(2, 23) \le S_{24}$  & $(A_5\times A_5):2^2\le A_{25}$  \\ 
$(S_5 \times S_5):2 \le S_{25}$ & $\mathsf{P\Gamma L}(2, 25) \le S_{26}$  & $\mathsf{ASL}(3, 3) \le A_{27}$  & $\mathsf{PSp}(4, 3):2 \le A_{27}$ \\ 
$\mathsf{AGL}(3, 3) \le S_{27}$  & $\mathsf{PSp}(6, 2) \le A_{28}$  & $\mathsf{PSL}(5, 2) \le A_{31}$  & $\mathsf{ASL}(5, 2) \le A_{32}$ \\ 
\hline
\end{tabular}
\medskip
\caption{The primitive groups $H$ of degree $n\le 47$ 
such that $|H|^6\ge n!/2$ or  $|H|^6\ge n!$, depending on whether $H\le A_n$ or not. }\label{primitivecandidates}
\end{center}
\end{table}

\begin{theorem}\label{A_n-smalldegree-done}
Suppose that $G= A_n$ or $S_n$ and that $G$ acts flag-transitively and point-primitively on a thick
generalised quadrangle $\Q$.  Then $G_\alpha$ does not act primitively on $\{1,\ldots,n\}$.
\end{theorem}

\begin{proof}
Suppose that $G_\alpha$ is primitive on $\{1,\ldots,n\}$.
By Lemma \ref{lemma:alternatinggroups}, we know that $n\le 47$. By Lemma \ref{lemma:power6}, $|G|
< |G_\alpha|^6$.  The possible pairs $(G_\alpha , G)$ are given in Table \ref{primitivecandidates}.  The following table
lists the only cases for which there are positive integers $s,t>1$ such that $|G:G_\alpha|=(s+1)(st+1)$,
$\sqrt{t}\le s\le t^2$ and $s+t\mid st(s+1)(t+1)$.

\begin{center}
\begin{tabular}{|c|c|c|c|c|}
\hline
$G_\alpha ,\ G$ & $s$ & $t$  & $s(t+1)$ & Nontrivial subdegrees\\
\hline
$\mathsf{PSL}(3, 2),\ A_{7}$  & 2 & 2 & 6 & $14$\\
$\mathsf{ASL}(3, 2),\ A_{8}$  & 2 & 2 & 6 & $14$\\
$M_{10},\ A_{10}$& 11 & 19 & 220 & $10,\ 10,\ 45,\ 90,\ 90,\ 90,\ 144,\ 240,\ 360,\ 720,\ 720$\\
$\mathsf{P\Gamma L}(2,9),\ S_{10}$& 11 & 19 & 220 & $20,\ 45,\ 90,\ 144,\ 180,\ 240,\ 360,\ 720,\ 720 $\\
$M_{11},\  A_{11}$& 11 & 19 & 220 & $110,\ 330,\ 495,\ 1584$\\
$M_{12},\  A_{12}$& 11 & 19 & 220 & $440,\ 495,\ 1584$\\
\hline
\end{tabular}
\end{center}
In each case, no subset of the subdegrees sums to $s(t+1)$, the cardinality of the neighbourhood
of a point, and hence none of these cases are realised.  \qed\end{proof}

%%%%%%%%%%%%%%%%%%%%%%%%%
%
%    Sporadics
%
%%%%%%%%%%%%%%%%%%%%%%%%%

\section{Sporadics}\label{section:sporadics}

From the Atlas of Finite Group Representations \cite{Atlas}, we can readily establish the following:

\begin{lemma}
Let $G$ be an almost simple group whose socle is a sporadic simple group and suppose that $G$ has a
maximal subgroup with index $(s+1)(st+1)$, for some $s,t\ge 2$ such that $s\le t^2$, $t\le s^2$ and
$s+t$ divides $st(st+1)$. Then $G$ is listed in Table~\ref{tab:sporadics}.
\end{lemma}

\begin{table}
\begin{center}
\begin{tabular}{lrrrr}
\hline
\multicolumn{1}{c}{$G$}&\multicolumn{1}{c}{$s$}&\multicolumn{1}{c}{$t$}&\multicolumn{1}{c}{$(s+1)(st+1)$}&\multicolumn{1}{c}{$G_\alpha$}\\
\hline
$\mathit{Co}_2$  & $161$ & $159$ & $4147200$ & $M_{23}$\\ 
$\mathit{Fi}_{24}'$  & $115$ & $23$ & $306936$ & $\mathit{Fi}_{23}$\\ 
$\mathit{Fi}_{24}'.2$  & $115$ & $23$ & $306936$ & $\mathit{Fi}_{23}\times 2$\\ 
$\mathit{Fi}_{22}$  & $25$ & $95$ & $61776$ & $O_8^+(2).3.2$\\ 
 & $35$ & $49$ & $61776$ & $O_8^+(2).3.2$\\ 
 & $39$ & $9$ & $14080$ & $O_7(3)$\\ 
$\mathit{Fi}_{22}.2$  & $25$ & $95$ & $61776$ & $O_8^+(2).S_3\times 2$\\ 
 & $35$ & $49$ & $61776$ & $O_8^+(2).S_3\times 2$\\ 
$\mathit{Fi}_{23}$  & $2991$ & $689$ & $6165913600$ & $[3^{10}].(L_3(3)\times 2)$\\ 
$\mathit{HN}$  & $149$ & $51$ & $1140000$ & $A_{12}$\\ 
$\mathit{HN}.2$  & $149$ & $51$ & $1140000$ & $S_{12}$\\ 
$J_1$  & $21$ & $9$ & $4180$ & $7:6$\\ 
$J_2$  & $9$ & $3$ & $280$ & $3.A_6.2_2$\\ 
 & $13$ & $11$ & $2016$ & $5^2:D_{12}$\\ 
$J_2.2$  & $9$ & $3$ & $280$ & $3.A_6.2^2$\\ 
 & $13$ & $11$ & $2016$ & $5^2:(4\times S_3)$\\ 
$J_3$ & $44$ & $22$ & $43605$ & $2^{2+4}.(3\times S_3)$\\ 
$J_3.2$ & $44$ & $22$ & $43605$ & $2^{2+4}:(S_3\times S_3)$\\ 
$M_{11}$ & $4$ & $8$ & $165$ & $2.S_4$\\ 
$\mathit{McL}$ & $8$ & $28$ & $2025$ & $M_{22}$\\ 
$\mathit{O'N}$ & $19$ & $323$ & $122760$ & $L_3(7).2$\\ 
$\mathit{Ru}$ & $9$ & $45$ & $4060$ & $\,^2F_4(2)$\\ 
 & $57$ & $57$ & $188500$ & $2^6:U_3(3):2$\\ 
$\mathit{Suz}$ & $41$ & $19$ & $32760$ & $U_5(2)$\\ 
 & $129$ & $191$ & $3203200$ & $3^{2+4}:2.(2^2\times A_4).2$\\ 
$\mathit{Suz}.2$ & $41$ & $19$ & $32760$ & $U_5(2).2$\\ 
 & $129$ & $191$ & $3203200$ & $3^{2+4}:2.(S_4\times D_8)$\\ 
\hline
\end{tabular}
\end{center}
\caption{Sporadic almost simple groups that might act primitively on a generalised quadrangle of order $(s,t)$.}
\label{tab:sporadics}
\end{table}

\begin{proof}
For all groups $G$ except the Fischer-Griess `Monster', we use the \textit{Atlas} \cite{Atlas} for a list of maximal subgroups of $G$ and
extract the indices satisfying the conditions of the lemma.
%Table~\ref{tab:sporadics}
%   lists only those indices $i$ of maximal subgroups of $G$ that are of the form $i=(s+1)(st+1)$
%for some $s,t\ge 2$ satisfying $\sqrt{t}\le s\le t^2$ and for which $s+t$ divides $st(s+1)(t+1)$.  
For the Monster $M$, there are $43$ known maximal subgroups, and none of them has index of the form required.  We know by
\cite{Bray:2006fk} that if $H$ is an unknown maximal subgroup of $M$, then $H$ is almost simple with
$\soc(H)\in\{L_2(13),U_3(4),U_3(8),Sz(8)\}$.  Hence we know the possible values for $|M:H|$, and we
find that none can be the number of points of a thick generalised quadrangle.  \qed\end{proof}

\begin{cor}\label{final-cor-sporadics}
No sporadic almost simple group acts primitively on both the points and lines of a thick
generalised quadrangle.
\end{cor}

\begin{proof}
There is only one group for which $(s,t)$ and $(t,s)$ are both listed, namely the Rudvalis group
with $s=t=57$. The nontrivial subdegrees for the transitive action of $G$ on the right cosets of $G_\alpha$ are
$$ 63,\ 756,\ 2016,\ 2016,\ 2016,\ 16128,\ 16128,\ 21504,\ 24192,\ 48384,\  55296 $$ (by computer).
In the collinearity graph of a generalised quadrangle of order $(57,57)$, the cardinality
of a neighbourhood is $57\times 58=3306$. However, there is no partition of
3306 into the subdegrees  of $G$.
\qed\end{proof}

\section{Acknowledgements}

This paper forms part of an Australian Research Council Discovery Project, which supports the first
author, and the second author holds an Australian Research Fellowship. The third author is supported
by a grant from the Natural Sciences and Engineering Research Council of Canada (NSERC).  She would
also like to thank the University of Western Australia for hospitality and support during her
sabbatical, when this work was undertaken.  The fifth author was supported by the University of
Western Australia as part of the ARC Federation Fellowship Project FF0776186.

%% JB: I put this back in because the journal style file doesn't put it in for some reason.
\section{References}

\bibliographystyle{abbrv}
\bibliography{references}

\end{document}